\documentclass[12pt]{amsart}

\usepackage{graphicx}
\usepackage{mathptmx}

\usepackage{amscd}
\usepackage{amsthm}
\usepackage{amsxtra}
\usepackage{a4wide}
\usepackage{latexsym}
\usepackage{amssymb}
\usepackage{amsfonts}
\usepackage{amsmath}
\usepackage{amsrefs}
\usepackage{mathrsfs}

\usepackage{upref}
\usepackage{txfonts}
\usepackage{graphicx}

\usepackage[bookmarksnumbered, colorlinks, plainpages]{hyperref}
\hypersetup{colorlinks=true,linkcolor=red, anchorcolor=green, citecolor=cyan, urlcolor=red, filecolor=magenta, pdftoolbar=true}

\usepackage[left=2cm, right=2cm, top=2cm, bottom=2cm]{geometry}

\allowdisplaybreaks

\theoremstyle{plain}
\newtheorem{thm}{Theorem}[section]
\newtheorem{lem}[thm]{Lemma}

\theoremstyle{definition}
\newtheorem{rem}[thm]{Remark}

\newtheorem{defi}[thm]{Definition}

\newtheorem{conv}[thm]{Convention}

\numberwithin{thm}{section}
\numberwithin{equation}{section}

\newcommand\ddfrac[2]{\frac{\displaystyle #1}{\displaystyle #2}}

\def\supp{\operatorname{supp}}

\def\esup{\operatornamewithlimits{ess\,sup}}

\def\Id{\operatorname{I}}

\def\ces{\operatorname{Ces}}
\def\cop{\operatorname{Cop}}

\def\ap{\approx}
\def\qq{\qquad}
\def\rw{\rightarrow}

\def\ls{\lesssim}
\def\gs{\gtrsim}

\def\hra{\hookrightarrow}

\def\M{\mathcal M}

\def\a{\alpha}
\def\b{\beta}

\def\la{\lambda}

\def\vp{\varphi}

\def\i{\infty}
\def\I{(0,\i)}
\def\t{\theta}

\def\R{\mathbb R}

\def\R{\mathbb R}

\def\M{\mathfrak M}

\def\W{{\mathcal W}}

\def\mp{{\mathfrak M}}

\def\a{\alpha}

\def\b{\beta}

\def\O{\Omega}

\def\la{\lambda}

\def\vp{\varphi}

\def\i{\infty}
\def\I{(0,\i)}
\def\t{\theta}

\def\dual{\,^{^{\bf c}}\!}

\begin{document}

\title{Embeddings between weighted Copson and Ces\`{a}ro function spaces}

\author[A. Gogatishvili]{Amiran Gogatishvili}
\address{Institute of Mathematics \\
    Academy of Sciences of the Czech Republic \\
    \v Zitn\'a~25 \\
    115~67 Praha~1, Czech Republic} \email{gogatish@math.cas.cz}

\author[R.Ch.Mustafayev]{Rza Mustafayev}
\address{Department of Mathematics \\ Faculty of Science and Arts \\ Kirikkale
    University \\ 71450 Yahsihan, Kirikkale, Turkey}
\email{rzamustafayev@gmail.com}

\author[T.~{\"U}nver]{Tu\v{g}\c{c}e {\"U}nver}
\address{Department of Mathematics \\ Faculty of Science and Arts \\ Kirikkale
    University \\ 71450 Yahsihan, Kirikkale, Turkey}
\email{tugceunver@gmail.com}

\thanks{The research of A. Gogatishvili was partly supported by the grants P201-13-14743S
    of the Grant Agency of the Czech Republic and RVO: 67985840, by
    Shota Rustaveli National Science Foundation grants no. 31/48
    (Operators in some function spaces and their applications in Fourier
    Analysis) and no. DI/9/5-100/13 (Function spaces, weighted
    inequalities for integral operators and problems of summability of
    Fourier series). The research of all authors was partly supported
    by the joint project between  Academy of Sciences of Czech Republic
    and The Scientific and Technological Research Council of Turkey}

\subjclass[2010]{Primary 46E30; Secondary 26D10.}

\keywords{Ces\`{a}ro and Copson function spaces, embeddings, iterated Hardy inequalities}

\begin{abstract}
    In this paper embeddings between weighted Copson function spaces  $\cop_{p_1,q_1}(u_1,v_1)$ and weighted Ces\`{a}ro function spaces
    $\ces_{p_2,q_2}(u_2,v_2)$ are characterized. In particular, two-sided estimates of the optimal constant $c$ in the inequality
    \begin{equation*}
    \bigg( \int_0^{\infty} \bigg( \int_0^t f(\tau)^{p_2}v_2(\tau)\,d\tau\bigg)^{\frac{q_2}{p_2}} u_2(t)\,dt\bigg)^{\frac{1}{q_2}} \le c \bigg( \int_0^{\infty} \bigg( \int_t^{\infty} f(\tau)^{p_1} v_1(\tau)\,d\tau\bigg)^{\frac{q_1}{p_1}} u_1(t)\,dt\bigg)^{\frac{1}{q_1}},
    \end{equation*}
    where $p_1,\,p_2,\,q_1,\,q_2 \in (0,\infty)$, $p_2 \le q_2$ and $u_1,\,u_2,\,v_1,\,v_2$ are weights on $\I$, are obtained. The most innovative part consists of the fact that possibly different parameters $p_1$ and $p_2$ and  possibly different inner weights $v_1$ and $v_2$ are allowed. The proof is based on the combination duality techniques with estimates of optimal constants of the embeddings between  weighted Ces\`{a}ro and Copson spaces and weighted Lebesgue spaces, which reduce the problem to the solutions of the iterated Hardy-type inequalities.
\end{abstract}

\maketitle


\section{Introduction}\label{introduction}

Many Banach spaces which play an important role in functional analysis
and its applications are obtained in a special way: the norms of these spaces are generated by positive sublinear operators and by $L_p$-norms.

In connection with Hardy and Copson operators
$$
(Pf)(x) : = \frac{1}{x} \int_0^x f(t)\,dt \qq \mbox{and} \qq (Qf)(x) : = \int_x^{\infty} \frac{f(t)}{t}\,dt,\qq (x > 0),
$$
the classical Ces\`{a}ro function space
$$
\ces(p) = \bigg\{ f:\, \|f\|_{\ces(p)} : =  \bigg( \int_0^{\infty} \bigg( \frac{1}{x} \int_0^x |f(t)|\,dt \bigg)^p\,dx \bigg)^{\frac{1}{p}} < \infty \bigg\},
$$
and the classical Copson function space
$$
\cop(p) = \bigg\{ f:\, \|f\|_{\cop(p)} : = \bigg( \int_0^{\infty} \bigg( \int_x^{\infty} \frac{|f(t)|}{t}\,dt \bigg)^p\,dx \bigg)^{\frac{1}{p}} < \infty \bigg\},
$$
where $1 < p \le \infty$, with the usual modifications if $p = \infty$, are of interest.

The classical Ces\`{a}ro function spaces $\ces(p)$ have been introduced in 1970 by Shiue \cite{shiue} and subsequently studied in \cite{hashus}.
These spaces have been defined analogously to the Ces\`{a}ro sequence spaces that appeared two years earlier in \cite{prog} when the Dutch Mathematical Society posted a problem to find a representation of their dual spaces. This problem was resolved by Jagers \cite{jagers} in 1974 who gave an explicit isometric description of the dual of Ces\`{a}ro sequence space. In \cite{syzhanglee}, Sy, Zhang and Lee gave a description of dual spaces of $\ces(p)$ spaces based on Jagers' result. In 1996 different, isomorphic description due to Bennett appeared in \cite{bennett1996}. For a long time, Ces\`{a}ro function spaces have not attracted a lot of attention contrary to their sequence counterparts. In fact there is quite rich literature concerning different topics studied in Ces\`{a}ro sequence spaces as for instance in \cites{CuiPluc,cmp,cuihud1999,cuihud2001,chencuihudsims,cuihudli}. However, recently in a series of papers \cites{astasmal2008,astasmal2009,astasmal2010,astasmalig10,astashkinmaligran11,asmal12,asmal13,astas5}, Astashkin and Maligranda started to study the structure of Ces\`{a}ro function spaces. Among others, in \cite{astasmal2009} they investigated dual spaces for $\ces (p)$ for $1 < p < \infty$. Their description can be viewed as being analogous to one given for sequence spaces in \cite{bennett1996} (For more detailed  information about history of classical Ces\`{a}ro spaces see recent survey paper \cite{asmalsurvey}).

In \cite[Theorem 21.1]{bennett1996} Bennett observes that the classical
Ces\`{a}ro function space and the classical Copson function space coincide for
$p > 1$. He also derives estimates for the norms of the
corresponding inclusion operators. The same result, with different
estimates, is due to Boas \cite{boas1970}, who in fact obtained the
integral analogue of the Askey-Boas theorem  \cite[Lemma
6.18]{boas1967} and \cite[Lemma]{askeyboas}. These results
generalized in \cite{grosse} using the blocking technique.

Let $A$ be any measurable subset of $\I$. By $\mp (A)$ we denote the
set of all measurable functions on $A$. The symbol $\mp^+ (A)$
stands for the collection of all $f\in\mp (A)$ which are
non-negative on $A$. The family of all weight functions (also called
just weights) on $A$, that is, measurable, positive and finite a.e.
on $A$, is given by $\W (A)$.

For $p\in (0,\i]$, we define the functional
$\|\cdot\|_{p,A}$ on $\mp (A)$ by
\begin{equation*}
\|f\|_{p,A} : =
\left\{\begin{array}{cl}
\bigg(\int_A |f(x)|^p \,dx \bigg)^{\frac{1}{p}} & \qq\mbox{if}\qq p<\i, \\
\esup_{A} |f(x)| & \qq\mbox{if}\qq p=\i.
\end{array}
\right.
\end{equation*}

If $w\in \W(A)$, then the weighted Lebesgue space
$L_p(w,A)$ is given by
\begin{equation*}
L_p(w,A) \equiv L_{p,w}(A) : = \{f\in \mp (A):\,\, \|f\|_{p,w,A} : = \|fw\|_{p,A} <
\i\},
\end{equation*}
and it is equipped with the quasi-norm $\|\cdot\|_{p,w,A}$. When $A = \I$, we often write simply $L_{p,w}$ and $L_p(w)$ instead of $L_{p,w}(A)$ and $L_p(w,A)$,
respectively.

We adopt the following usual conventions.
\begin{conv}\label{Notat.and.prelim.conv.1.1}
{\rm (i)} Throughout the paper we put $0/0 = 0$, $0 \cdot (\pm \i)
= 0$ and $1 / (\pm\i) =0$.

{\rm (ii)} We put
$$
p' : = \left\{\begin{array}{cl} \frac p{1-p} & \text{if} \quad 0<p<1,\\
\infty &\text{if}\quad p=1, \\
\frac p{p-1}  &\text{if}\quad 1<p<\infty,\\
1  &\text{if}\quad p=\infty.
\end{array}
\right.
$$

{\rm (iii)} If $I = (a,b) \subseteq \R$ and $g$ is a monotone function on $I$, then by $g(a)$ and $g(b)$ we mean the limits $\lim_{x\rw a+}g(x)$ and $\lim_{x\rw b-}g(x)$, respectively.
\end{conv}

To state our results we use the notation $p \rw q$ for
$0 < p,\,q \le \infty$ defined by
$$
\frac{1}{p \rw q} = \frac{1}{q} - \frac{1}{p} \qq \mbox{if} \qq q <
p,
$$
and $p \rw q = \infty$ if $q \ge p$ (see, for instance, \cite[p.
30]{grosse}).

Throughout the paper, we always denote by $c$ and $C$ a positive constant, which is independent of main parameters but it may vary from line to line.
However a constant with subscript or superscript such as $c_1$ does not change in different occurrences.
By $a\lesssim b$, ($b\gtrsim a$) we mean that $a\leq \la b$, where $\la>0$ depends on inessential parameters. If $a\lesssim b$ and $b\lesssim a$, we write $a\approx b$ and say that $a$ and $b$ are equivalent.  We will denote by $\bf 1$ the function ${\bf 1}(x) = 1$, $x \in \R$.

Given two quasi-normed vector spaces $X$ and $Y$, we write $X=Y$
if $X$ and $Y$ are equal in the algebraic and the topological
sense (their quasi-norms are equivalent). The symbol
$X\hookrightarrow Y$ ($Y \hookleftarrow X$) means that $X\subset
Y$ and the natural embedding $\Id$ of $X$ in $Y$ is continuous,
that is, there exist a constant $c > 0$ such that $\|z\|_Y \le
c\|z\|_X$ for all $z\in X$. The best constant of the embedding
$X\hookrightarrow Y$ is $\|\Id\|_{X \rw Y}$.

The weighted Ces\`{a}ro and Copson function spaces are defined as follows:
\begin{defi}\label{defi.2.1}
Let  $0 <p, q \le \infty$, $u \in \mp^+ (I)$, $v\in \W(I)$. The weighted Ces\`{a}ro and Copson spaces are defined by
\begin{align*}
\ces_{p,q} (u,v) : & = \bigg\{ f \in \mp^+ \I: \|f\|_{\ces_{p,q} (u,v)} : = \big\| \|f\|_{p,v,(0,\cdot)}
\big\|_{q,u,\I} < \i \bigg\}, \\
\intertext{and} \cop_{p,q} (u,v) : & = \bigg\{ f \in \mp^+ \I: \|f\|_{\cop_{p,q} (u,v)} : = \big\| \|f\|_{p,v,(\cdot,\i)}   \big\|_{q,u,\I} < \i \bigg\},
\end{align*}
respectively.
\end{defi}
Many function spaces from the literature, in particular from Harmonic Analysis, are covered by the spaces $\ces_{p,q}(u,v)$ and $\cop_{p,q}(u,v)$. Let us only mention the Beurling algebras $A^p$ and $A^*$, see \cite{gil1970,johnson1974,belliftrig}.

Note that the function spaces $C$ and $D$ defined by Grosse-Erdmann in \cite{grosse} are related with our definition in the following way:
$$
\ces_{p,q}(u,v) = C(p,q,u)_v \qq \mbox{and} \qq \cop_{p,q}(u,v) = D(p,q,u)_v.
$$

We use the notations $\ces_p(u) : = \ces_{1,p}(u,{\bf 1})$ and $\cop_p(u) : = \cop_{1,p}(u,{\bf 1})$. Obviously, $\ces(p) = \ces_p (x^{-1})$ and $\cop(p) = \cop_p (x^{-1})$. In \cite{kamkub}, Kami{\'n}ska and Kubiak computed the dual norm of the
Ces\`{a}ro function space $\ces_{p}(u)$, generated by $1 < p < \infty$ and an arbitrary positive weight $u$. A description presented in \cite{kamkub} resembles the approach of Jagers \cite{jagers} for sequence spaces.

Our principal goal in this paper is to investigate the embeddings between weighted Copson and Ces\`{a}ro function spaces and vice versa, that is, the embeddings
\begin{align}
\cop_{p_1,q_1}(u_1,v_1) & \hra \ces_{p_2,q_2}(u_2,v_2), \label{mainemb1}\\
\ces_{p_1,q_1}(u_1,v_1) & \hra \cop_{p_2,q_2}(u_2,v_2). \label{mainemb2}
\end{align}
This is a very difficult and technically complicated task. We
develop an approach consisting of a duality argument combined with
estimates of optimal constants of the embeddings between  weighted
Ces\`{a}ro and Copson spaces and weighted Lebesgue spaces, that is,
\begin{align}
L_s(w) & \hra \ces_{p,q}(u,v), \label{emb1}\\
L_s(w) & \hra \cop_{p,q}(u,v), \label{emb2}\\
L_s(w) & \hookleftarrow \ces_{p,q}(u,v), \label{emb3}\\
L_s(w) & \hookleftarrow \cop_{p,q}(u,v), \label{emb4}
\end{align}
which reduce the problem to the solutions of the iterated Hardy-type inequalities \eqref{eq.4.11}. In order to characterize embeddings \eqref{emb1} - \eqref{emb4}, we are going to use the direct and reverse Hardy-type inequalities. Note that embeddings \eqref{mainemb1} - \eqref{mainemb2} contain embeddings \eqref{emb1} - \eqref{emb4} as a special case. Indeed, for instance, if $p = q$ and $v(x) = w(x) / \|u\|_{p,(x,\infty)}$, then $\ces_{p,q}(u,v) = L_p(w)$. Similarly, if $p = q$ and $v(x) = w(x) / \|u\|_{p,(0,x)}$, then $\cop_{p,q}(u,v) = L_p(w)$. Moreover, by the change of variables $x = {1} / {t}$ it is easy to see that \eqref{mainemb2} is equivalent to the embedding
$$
\cop_{p_1,q_1}(\tilde{u}_1,\tilde{v}_1) \hra
\ces_{p_2,q_2}(\tilde{u}_2,\tilde{v}_2),
$$
where $\tilde{u}_i (t) = t^{- {2} / {q_i}}u_i\big({1} / {t}\big)$, $\tilde{v}_i (t) = t^{- {2} / {p_i}} v_i\big({1} / {t}\big)$,  $i=1,2$, $t > 0$. This note allows us to concentrate our attention on characterization of \eqref{mainemb1}. On the negative side of things we have to admit that the duality approach works only in the case when, in \eqref{mainemb1} - \eqref{mainemb2}, one has $p_2 \le q_2$. Unfortunately, in the case when $p_2 > q_2$ the characterization of these embeddings remain open.

It should be noted that none of the above would ever have existed if it wasn't for the (now classical) well-known characterizations of weights for which the Hardy inequality holds. This subject, which is, incidentally, exactly one hundred years old, is absolutely indispensable   in this part of mathematics. In our proof below such results will be heavily used, as well as the more recent characterizations of the weighted reverse inequalities (cf. \cite{ego2008} and \cite{mu2015}).

It is mentioned in \cite[p. 30]{grosse} that multipliers between Ces\`{a}ro and Copson
spaces  are more difficult to treat. It is worth to mention that by using
characterizations of \eqref{mainemb1} - \eqref{mainemb2} it is
possible to  give the solution to the multiplier problem between
weighted Ces\`{a}ro and Copson function spaces, and we are going to
present it in the future paper.

In particular, we obtain two-sided estimates of the optimal constant $c$ in the inequality
\begin{equation}\label{emb.as.ineq.}
\bigg( \int_0^{\infty} \bigg( \int_0^t f(\tau)^{p_2}v_2(\tau)\,d\tau\bigg)^{\frac{q_2}{p_2}} u_2(t)\,dt\bigg)^{\frac{1}{q_2}} \le c \bigg( \int_0^{\infty} \bigg( \int_t^{\infty} f(\tau)^{p_1} v_1(\tau)\,d\tau\bigg)^{\frac{q_1}{p_1}} u_1(t)\,dt\bigg)^{\frac{1}{q_1}},
\end{equation}
where $p_1,\,p_2,\,q_1,\,q_2 \in (0,\infty)$, $p_2 \le \min \{p_1,\, q_2\}$ and $u_1,\,u_2,\,v_1,\,v_2$ are weights on $\I$ (It is shown in  Lemma \ref{triviality} that inequality \eqref{emb.as.ineq.} holds true only for trivial functions $f$ when $p_1 < p_2$ for any $q_1,\,q_2 \in (0,\infty]$). The most innovative part consists of the fact that possibly different parameters $p_1$ and $p_2$ and  possibly different inner weights $v_1$ and $v_2$ are allowed. Note that \eqref{emb.as.ineq.} was characterized in the particular cases, when $p_1 = p_2 = 1$, $q_1 = q_2 = p > 1$, $u_1(t) = t^{\b p - 1}$, $u_2(t) = t^{-\a p - 1}$, $v_1(t) = t^{-\b - 1}$, $v_2(t) = t^{\a - 1}$, $t > 0$, where $\a > 0$ and $\b > 0$, in \cite[p. 61]{boas1970}, and, when $p_1 = p_2 = 1$, $q_1 = p$, $q_2 = q$,  $u_1(t) = v(t)$, $u_2(t) = t^{-q}w(t)$, $v_1(t) = t^{-1}$, $v_2 (t) = 1$, $t > 0$, where $0 < p \le \infty$, $1 \le q \le \infty$ and $v,\,w$ are weight functions on $\I$, in \cite[Theorem 2.3]{cargogmarpick}, respectively.

The paper is organized as follows. We start with notation and
preliminary results in Section~\ref{s.2}. In Section \ref{s3} we recall characterizations of direct and reverse weighted Hardy-type inequalities.  Solutions of embeddings \eqref{emb1} - \eqref{emb2} and \eqref{emb3} - \eqref{emb4} are given in Sections \ref{s4} and \ref{s5}, respectively. In Section \ref{s6} we recall characterizations of weighted iterated Hardy-type inequalities. The characterization of the embeddings between  weighted Ces\`{a}ro and Copson spaces are obtained in Section \ref{s7}.


\section{Notations and preliminaries}\label{s.2}

Let $A,\,B$ be some sets and $\vp,\,\psi$ be non-negative functions defined on $A \times B$ (It may happen that $\vp (\a,\b)= \i$ or $\psi (\a,\b) = \i$ for some $\a \in A$, $\b \in B$). We
say that $\vp$ is dominated by $\psi$ (or $\psi$ dominates $\vp$)
on $A \times B$ uniformly in $\a \in A$ and write
$$
\vp (\a,\b) \ls \psi (\a,\b) \quad \mbox{uniformly in} \quad \a \in A,
$$
or
$$
\psi (\a,\b) \gs \vp (\a,\b) \quad \mbox{uniformly in} \quad \a \in A,
$$
if for each $\b \in B$ there exists $C(\b) > 0$ such that
$$
\vp (\a,\b) \le C(\b) \psi (\a,\b)
$$
for all $\a \in A$. We also say that $\vp$ is equivalent to $\psi$
on $A \times B$ uniformly in $\a \in A$ and write
$$
\vp (\a,\b) \ap \psi (\a,\b) \quad \mbox{uniformly in} \quad \a \in A,
$$
if $\vp$ and $\psi$ dominate each other on $A \times B$ uniformly
in $\a \in A$.

We need the following auxiliary results.
\begin{lem}\label{NontrivialCesCopLemma} Let $0<p,q\le \i$, $v\in \W\I$ and let $u\in \M^+\I$.
$\ces_{p,q} (u,v)$ and $\cop_{p,q} (u,v)$ are non-trivial, i.e. consists not only of functions equivalent to $0$ on $\I$, if and only if
\begin{equation*}
\|u\|_{q,(t,\i)} < \i, \qq \mbox{for some} \qq t >0,
\end{equation*}
and
\begin{equation*}
\|u\|_{q,(0,t)} < \i,\qq \mbox{for some} \qq t >0,
\end{equation*}
respectively.
\end{lem}
\begin{proof}
\textit{Sufficiency.} Let $u\in \M^+\I$ be such that $\|u\|_{q,(t,\i)}=\i$ for all $t>0$. Assume that $f\not = 0$ a.e.
Then $\|f\|_{p,v,(0,t_0)}>0$ for some $t_0>0$.
\begin{align*}
\|f\|_{\ces_{p,q}(u,v)} \geq  \big\| \|f\|_{p,v,(0,\cdot)}
\big\|_{q,u,(t_0,\i)} \geq \|{\bf{1}}\|_{q,u,(t_0,\i)}
\|f\|_{p,v,(0,t_0)} = \|u\|_{q,(t_0,\i)} \|f\|_{p,v,(0,t_0)}.
\end{align*}
Hence $\|f\|_{\ces_{p,q}(u,v)}=\i$. Consequently, if $
\|f\|_{\ces_{p,q}(u,v)} < \i $, then $f = 0$ a.e., that is,
$\ces_{p,q}(u,v)=\{0\}$.

\textit{Necessity.}  Assume that $\|u\|_{q,(t,\i)}<\i$ for some
$t>0$. If $f\in L_p(v)$ such that $\supp f \subset (\tau,\infty)$
for some $\tau \ge t$, then $f \in \ces_{p,q}(u,v)$. Indeed:
\begin{align*}
\|f\|_{\ces_{p,q}(u,v)} = \big\| \|f\|_{p,v,(0,\cdot)}
\big\|_{q,u,(\tau, \i)} \leq \|{\bf{1}}\|_{q,u,(\tau, \i)}
\|f\|_{p,v,(0,\i)} = \|u\|_{q,(\tau, \i)} \|f\|_{p,v,(0,\i)} < \i.
\end{align*}

The same conclusion can be deduced for the Copson spaces.
\end{proof}
\begin{lem}\label{shrinkagelemma}
If $\|u\|_{q,(t_1,\infty)}=\infty$ for some $t_1 > 0$, then
$$
f\in \ces_{p,q}(u,v) \Rightarrow f = 0 \quad \mbox{a.e. on} \quad (0,t_1).
$$
If $\|u\|_{q,(0,t_2)}=\infty$ for some $t_2 > 0$, then
$$
f\in \cop_{p,q}(u,v) \Rightarrow f = 0 \quad \mbox{a.e. on} \quad (t_2,\infty).
$$
\end{lem}
\begin{proof}
Assume that $\|u\|_{q,(t_1,\infty)}=\infty$ for some $t_1 > 0$ and let $f\in \ces_{p,q}(u,v)$. Then
\begin{align*}
\|f\|_{\ces_{p,q}(u,v)} \geq \big\| \|f\|_{p,v,(0,\cdot)} \big\|_{q,u,(t_1,\infty)} \geq  \|u\|_{q,(t_1,\infty)} \|f\|_{p,v,(0,t_1)}.
\end{align*}
Therefore, $\|f\|_{p,v,(0,t_1)}=0$. Hence, $f = 0$ a.e. on $(0,t_1)$.

Assume now that $\|u\|_{q,(0,t_2)}=\infty$ for some $t_2 > 0$ and let $f\in \cop_{p,q}(u,v)$. Then
\begin{align*}
\|f\|_{\cop_{p,q}(u,v)} \geq \big\| \|f\|_{p,v,(\cdot,\infty)} \big\|_{q,u,(0,t_2)} \geq  \|u\|_{q,(0,t_2)} \|f\|_{p,v,(t_2,\infty)}.
\end{align*}
Consequently, $\|f\|_{p,v,(t_2,\infty)}=0$. This yields that $f = 0$ a.e. on $(t_2,\infty)$.
\end{proof}
\begin{rem}
In view of Lemmas~\ref{NontrivialCesCopLemma} and \ref{shrinkagelemma}, it is enough to take $u\in \M^+(0,\infty)$ such that $\|u\|_{q,(t,\infty)}<\infty$ for all $t>0$, when considering $\ces_{p,q}(u,v)$ spaces. Similarly, it is enough to take $u\in \M^+(0,\infty)$ such that $\|u\|_{q,(0,t)}<\infty$ for all $t>0$, when considering $\cop_{p,q}(u,v)$ spaces.
\end{rem}

\begin{defi}
Let $0<q\leq \i$. We denote by $\O_q$ the set of all functions $u \in \mp^+ \I$ such that
$$
0<\|u\|_{q,(t,\i)}<\i,~~ t>0,
$$
and by $\dual{\O}_q$ the set of all functions $u \in \mp^+ \I$ such that
$$
0<\|u\|_{q,(0,t)}<\i,~~ t>0.
$$
\end{defi}
Let $v \in \W\I$. It is easy to see that $\ces_{p,q} (u,v)$ and $\cop_{p,q} (u,v)$ are quasi-normed vector spaces when $u \in \O_q$ and $u \in \dual{\O}_q$, respectively.

Note that $\ces_{p,p}(u,v)$ and $\cop_{p,p}(u,v)$ coincide with some weighted Lebesgue spaces.

\begin{lem}\label{Cespp}
Let $0<p\le \i$, $u\in \O_p$ and $v\in \W\I$. Then $\ces_{p,p}(u,v)=L_p(w)$,
where
\begin{equation}\label{WeightCesaro}
w(x) := v(x)\|u\|_{p,(x,\i)}, ~ x > 0.
\end{equation}
\end{lem}
\begin{proof}
Assume first that $p<\i$. Applying Fubini's Theorem, we have
\begin{align*}
\|f\|_{\ces_{p,p}(u,v)}&=\bigg(\int_0^\i u^p(t) \int_0^t f(\tau)^p v(\tau)^p \,d\tau \,dt\bigg)^{\frac{1}{p}}\\
&= \bigg(\int_0^\i  f(\tau)^p v(\tau)^p \int_{\tau}^\i u(t)^p \, dt\, d\tau\bigg)^{\frac{1}{p}}\\
&=\|f\|_{p,w,\I},
\end{align*}
where $w$ is defined by \eqref{WeightCesaro}. If $p=\i$, by exchanging suprema, we have
\begin{align*}
\|f\|_{\ces_{\i,\i}(u,v)}&=\esup_{t\in \I} u(t) \esup_{\tau\in(0,t)} f(\tau)v(\tau)\\
&=\esup_{t\in \I} f(t)v(t) \esup_{\tau \in (t,\i)} u(\tau)\\
&= \|f\|_{\i,w,\I}.
\end{align*}
\end{proof}
\begin{lem}\label{Coppp}
Let $0<p\le \i$, $u\in \dual{\O_p}$ and $v\in \W\I$. Then
$\cop_{p,p}(u,v)=L_p(w)$, where
\begin{equation}\label{WeightCopson}
w(x) := v(x)\|u\|_{p,(0,x)}, ~ x > 0.
\end{equation}
\end{lem}
\begin{proof}
This follows by the same method as in Lemma~\ref{Cespp}.
\end{proof}



\section{Some Hardy-type inequalities}\label{Hardy Type Inequalities}\label{s3}

In this section we recall characterizations of direct and reverse weighted Hardy-type inequalities.
Denote by
$$
(Hf) (t) : = \int_0^t f(x)\,dx,  \qq
(H^*f) (t) : = \int_t^{\i} f(x)\,dx,  \qq f \in \mp^+\I,\qq t \ge 0.
$$
The well-known two-weight Hardy-type inequalities
\begin{equation}\label{eq.2.2}
\|Hf\|_{q,w,\I}\leq c \|f\|_{p,v,\I}
\end{equation}
and
\begin{equation}\label{eq.2.200}
\|H^*f\|_{q,w,\I}\leq c \|f\|_{p,v,\I}
\end{equation}
for all non-negative measurable functions $f$ on $(0,\infty)$, where
$0 < p,\,q \le \infty$ with $c$ being a constant independent of $f$,
have a broad variety of applications and represents now a basic tool
in many parts of mathematical analysis, namely in the study of
weighted function inequalities. For the results, history and
applications of this problem, see \cites{ok,kp, kufmalpers}.

\begin{thm}\label{HardyIneq}
Let $1 \le p \le \i$, $0 < q \le \i$, $v,\,w \in \mp^+\I$.
Then inequality \eqref{eq.2.2}
holds for all $f \in \mp^+\I$ if and only if $A(p,q) < \i$, and the
best constant in \eqref{eq.2.2}, that is,
\begin{equation*}
B(p,q) : = \sup_{f\in \mp^+\I} \|Hf\|_{q,w,\I} /
\|f\|_{p,v,\I}
\end{equation*}
satisfies $B(p,q) \ap A(p,q)$, where

{\rm (i)} for $p \le q$,
$$
A(p,q): = \sup_{t \in \I} \big\|v^{-1}\big\|_{p',(0,t)} \|w\|_{q,(t,\i)} \,;
$$

{\rm (ii)} for $q<p$  and $\frac{1}{r} = \frac{1}{q} -  \frac{1}{p}$,
$$
A(p,q) : = \bigg(\int_0^{\infty} \big\|v^{-1}\big\|_{p',(0,t)}^{r} d \bigg(- \|w\|_{q,(t,\i)}^r\bigg)\bigg)^{\frac{1}{r}}.
$$
\end{thm}

\begin{thm}\label{CopsonIneq}
Let $1 \le p \le \i$, $0 < q \le \i$, $v,\,w \in \mp^+\I$.
Then inequality \eqref{eq.2.200}
holds for all $f \in \mp^+\I$ if and only if $A^*(p,q) < \i$, and
the best constant in \eqref{eq.2.200}, that is,
\begin{equation*}
B^*(p,q) : = \sup_{f\in \mp^+\I} \|H^*f\|_{q,w,\I} /
\|f\|_{p,v,\I}
\end{equation*}
satisfies $B^*(p,q) \ap A^*(p,q)$. Here

{\rm (i)} for $p \le q$,
$$
A^*(p,q): = \sup_{t \in \I} \big\|v^{-1}\big\|_{p',(t,\infty)} \|w\|_{q,(0,t)} \,;
$$

{\rm (ii)} for $q < p$  and $\frac{1}{r} = \frac{1}{q} -  \frac{1}{p}$,
$$
A^*(p,q) : = \bigg(\int_0^{\infty} \big\|v^{-1}\big\|_{p',(t,\infty)}^{r} d \bigg(\|w\|_{q,(0,t)}^r\bigg)\bigg)^{\frac{1}{r}}.
$$
\end{thm}

\begin{thm}\label{HardyforSupremal}
Let $0 < q < \i$, $v,\,w \in \mp^+\I$. Denote by
$$
(Sf) (t) : = \esup_{x \in (0,t)} f(x),  \qq f \in \mp^+\I,\qq t \ge
0.
$$
Then the inequality
\begin{equation*}
\|Sf\|_{q,w,\I}\leq c \|f\|_{\i,v,\I}
\end{equation*}
holds for all $f \in \mp^+\I$ if and only if
$$
\bigg(\int_0^{\infty} \big\|v^{-1}\big\|_{\infty,(0,t)}^q d \bigg(- \|w\|_{q,(t,\i)}^q\bigg)\bigg)^{\frac{1}{q}} <\i,
$$
and
\begin{equation*}
\sup_{f\in \mp^+\I} \|Sf\|_{q,w,\I} /
\|f\|_{\i,v,\I} \ap \bigg(\int_0^{\infty} \big\|v^{-1}\big\|_{\infty,(0,t)}^q d \bigg(- \|w\|_{q,(t,\i)}^q\bigg)\bigg)^{\frac{1}{q}}.
\end{equation*}
\end{thm}

\begin{thm}\label{CopsonforSupremal}
Let $0 < q < \i$, $v,\,w \in \mp^+\I$. Denote by
$$
(S^*f) (t) : = \esup_{x \in (t,\i)} f(x),  \qq f \in \mp^+\I,\qq t
\ge 0.
$$
Then the inequality
\begin{equation*}
\|S^*f\|_{q,w,\I}\leq c \|f\|_{\i,v,\I}
\end{equation*}
holds for all $f \in \mp^+\I$ if and only if
$$
\bigg(\int_0^{\infty} \big\|v^{-1}\big\|_{\infty,(t,\infty)}^q d \bigg(\|w\|_{q,(0,t)}^q\bigg)\bigg)^{\frac{1}{q}} < \i,
$$
and
\begin{equation*}
\sup_{f\in \mp^+\I} \|S^*f\|_{q,w,\I} /
\|f\|_{\i,v,\I} \ap \bigg(\int_0^{\infty} \big\|v^{-1}\big\|_{\infty,(t,\infty)}^q d \bigg(\|w\|_{q,(0,t)}^q\bigg)\bigg)^{\frac{1}{q}}.
\end{equation*}
\end{thm}

For the convenience of the reader we repeat the relevant material
from \cite{ego2008} and \cite{mu2015} without proofs, thus making our exposition
self-contained.
Let $\vp$ be a non-decreasing and finite function on the interval $I
: = (a,b)\subseteq \R$. We assign to $\vp$ the function $\la$
defined on subintervals of $I$ by
\begin{align}
\la ([y,z]) & = \vp(z+) - \vp(y-), \notag\\
\la ([y,z)) & = \vp(z-) - \vp(y-), \label{Notat.and.prelim.eq.1.4}\\
\la ((y,z]) & = \vp(z+) - \vp(y+), \notag\\
\la ((y,z)) & = \vp(z-) - \vp(y+). \notag
\end{align}
The function $\la$ is a non-negative, additive and regular
function of intervals. Thus (cf. \cite{r}, Chapter 10), it admits
a unique extension to a non-negative Borel measure $\la$ on $I$.

Note also that the associated Borel measure can be determined,
e.g., only by putting
$$
\la ([y,z]) = \vp(z+) - \vp(y-) \qq \mbox{for any}\qq [y,z]\subset
I
$$
(since the Borel subsets of $I$ can be generated by subintervals
$[y,z]\subset I$).

If $J\subseteq I$, then the Lebesgue-Stieltjes integral $\int_J
f\,d\vp$ is defined as $\int_J f\,d\la$. We shall also use the
Lebesgue-Stieltjes integral $\int_J f\,d\vp$ when $\vp$ is a
non-increasing and finite on the interval $I$. In such a case we
put
$$
\int_J f\,d\vp : = - \int_J f\,d(-\vp).
$$

We adopt the following conventions.

\begin{conv} \label{conv:3.3}
Let $I=(a,b)\subseteq \R$, $f:I\to [0,\infty]$ and $h:I\to
[0,\i]$. Assume that $h$ is non-decreasing and left-continuous on
$I$. If $h:I\to [0,\infty)$, then the symbol $\int_{I}f\,dh$ means
the usual Lebesgue-Stieltjes integral (with the measure $\la$
associated to $h$ is given by $\la([\a,\b))= h(\b)-h(\a)$ if $[\a,
\b)\subset (a,b)$ -- cf. \eqref{Notat.and.prelim.eq.1.4}).
However, if $h = \infty$ on some subinterval $(c,b)$  with $c\in
I$, then we define $\int_{I}f\,dh$ only if $f=0$ on $[c,b)$ and we
put
$$
\int_{I}f\,dh=\int_{(a,c)}f\,dh.
$$
\end{conv}

\begin{conv}
Let $I=(a,b)\subseteq \R$, $f:I\to [0,+\infty]$ and $h:I\to [-\infty,0]$. Assume that $h$ is non-decreasing and right-continuous on $I$. If $h:I\to (-\infty,0]$, then the symbol $\int_{I}f\,dh$ means the usual Lebesgue-Stieltjes integral.
However, if $h= -\infty$ on some subinterval $(a,c)$  with $c\in I$, then we define  $\int_{I}f\,dh$ only if $f=0$ on $(a,c]$ and we put
$$
\int_{I}f\,dh=\int_{(c,b)}f\,dh.
$$
\end{conv}

\begin{thm}\cite[Theorems 5.1 and 5.4]{ego2008}\label{ReverseHardyIneq}
Let  $w \in \mp^+\I$ and $u \in \mp^+\I$ be such that $\|u\|_{q,(t,\i)} <\infty$ for all $t\in (0,\i)$.

{\rm (i)} Assume that $0< q\le p\le1$. Then
\begin{equation} \label{5.1}
\Vert g\Vert _{p,w,\I}\leq c \Vert Hg\Vert _{q,u,\I}
\end{equation}
holds for all $g \in \mp^+\I$ if and only if
\begin{equation} \label{5.2}
C(p,q) : =\sup _{t\in (0,\i)}\| w \|_{{p'},(t,\i)} \|u\|_{q,(t,\i)}^{-1} <\infty.
\end{equation}
The best possible constant in \eqref{5.1}, that is,
\begin{equation*}
D(p,q) : = \sup _{g \in \mp^+\I}\| g \|_{p,w,\I} / \| Hg\|_{q,u,\I}
\end{equation*}
satisfies $D(p,q)\approx C(p,q)$.

{\rm (ii)} Let $0< p\le1$, $p<q\le\infty$ and $\frac{1}{r} = \frac{1}{p} - \frac{1}{q}$. Then \eqref{5.1}  holds if and only if
$$
C(p,q) : =\bigg( \int_{\I} \| w \|_{{p'},(t,\i)}^{r}\,d\bigg(\|u\|_{q,(t-,\i) }^{-r}\bigg) \bigg) ^{\frac{1}{r}}+\frac{\|w \|_{{p'},\I}} {\|u\|_{q,\I}}<\infty,
$$
and $D(p,q)\approx C(p,q)$.

\end{thm}

\begin{thm}\cite[Theorems 4.1 and 4.4]{ego2008}\label{ReverseCopsonIneq}
Let  $w \in \mp^+\I$ and $u \in \mp^+\I$ be such that $\|u\|_{q,(0,t)} <\infty$ for all $t\in (0,\i)$.

{\rm (i)} Assume that $0< q\le p\le1$. Then
\begin{equation} \label{5.100}
\Vert g \Vert _{p,w,\I}\leq c \Vert H^*g\Vert_{q,u,(0,\i)}
\end{equation}
holds for all $g \in \mp^+\I$ if and only if
\begin{equation} \label{5.200}
C^*(p,q) : =\sup _{t\in (0,\i)}\| w \|_{{p'},(0,t)} \|u\|_{q,(0,t)}^{-1} <\infty .
\end{equation}
The best possible constant in \eqref{5.100}, that is,
\begin{equation*}
D^*(p,q) : = \sup _{g \in \mp^+\I}\| g \|_{p,w,\I} / \|H^*g\|_{q,u,\I}
\end{equation*}
satisfies $D^*(p,q)\approx C^*(p,q)$.

{\rm (ii)} Let $0< p\le1$, $p<q\le\infty$ and $\frac{1}{r} = \frac{1}{p} - \frac{1}{q}$. Then \eqref{5.100}  holds if and only if
$$
C^*(p,q) : =\bigg( \int_{\I} \| w \|_{p',(0,t)}^{r}\,d\bigg(-\|u\|_{q,(0,t+) }^{-r}\bigg) \bigg) ^{\frac{1}{r}}+\frac{\|w \|_{{p'},\I}} {\|u\|_{q,\I}}<\infty,
$$
and $D^*(p,q)\approx C^*(p,q)$.

\end{thm}

\begin{rem} \label{R:5.5}
Let $q<\infty$ in Theorems~\ref{ReverseHardyIneq} and \ref{ReverseCopsonIneq}. Then
$$
\| u\|_{q,(t-,\i)}=\| u\|_{q,(t,\i)}\quad \mbox{and} \qq \|u\|_{q,(0,t+) } = \|u\|_{q,(0,t)} \qq\text{for all} \quad t\in \I,
$$
which implies that
$$
C(p,q) = \bigg( \int_{(0,\infty)} \Vert w \Vert _{{p'},(t,\i)}^{r} \,d\bigg(\Vert u\Vert _{q,(t,\i)}^{-r}\bigg) \bigg) ^{\frac1{r}}+\frac{\Vert w \Vert _{{p'},\I}} {\Vert u\Vert_{q,\I}},
$$
and
$$
C^*(p,q) : =\bigg( \int_{\I} \| w \|_{p',(0,t)}^{r}\,d\bigg(-\|u\|_{q,(0,t) }^{-r}\bigg) \bigg) ^{\frac{1}{r}} +\frac{\|w\|_{p',\I}} {\|u\|_{q,\I}}<\infty.
$$
\end{rem}

\begin{thm}\cite[Theorem 4.1]{mu2015}\label{ReverseHardySupremal}
Let  $w \in \mp^+\I$ and $u \in \mp^+\I$ be such that $\|u\|_{q,(t,\i)} <\infty$ for all $t\in (0,\i)$.

{\rm (i)} Assume that $0< q\le p\le \i$. Then
\begin{equation} \label{9.1}
\Vert g \Vert _{p,w,\I}\leq c \Vert Sg\Vert_{q,u,(0,\i)}
\end{equation}
holds for all $g \in \mp^+\I$ if and only if
\begin{equation} \label{9.2}
E(p,q) : = \sup_{t \in \I} \| w \|_{{p},(t,\i)} \|u\|_{q,(t,\i)}^{-1}  <\infty .
\end{equation}
The best possible constant in \eqref{9.1}, that is,
\begin{equation*}
F(p,q) : = \sup _{g \in \mp^+\I}\| g \|_{p,w,\I} / \| Sg\|_{q,u,\I}
\end{equation*}
satisfies $F(p,q)\approx E(p,q)$.\\

{\rm (ii)} Let $0 < p < q \le +\i$ and $\frac{1}{r} = \frac{1}{p} - \frac{1}{q}$. Then \eqref{9.1}  holds if and only if
$$
E(p,q) : =\bigg( \int_{\I} \| w \|_{{p},(t,\i)}^{r} \, d\bigg(\|u\|_{q,(t,\i)}^{-r} \bigg) \bigg)^{\frac{1}{r}} +\frac{\|w\|_{{p},\I}} {\|u\|_{q,\I}}<\infty,
$$
and $F(p,q)\approx E(p,q)$.

\end{thm}

\begin{thm}\cite[Theorem~3.4]{mu2015}\label{ReverseCopsonSupremal}
Let  $w \in \mp^+\I$ and $u \in \mp^+\I$ be such that $\|u\|_{q,(0,t)} <\infty$ for all $t\in (0,\i)$.

{\rm (i)} Assume that $0< q\le p\le \i$. Then
\begin{equation} \label{9.100}
\Vert g \Vert _{p,w,\I}\leq c \Vert S^*g\Vert_{q,u,(0,\i)}
\end{equation}
holds for all $g \in \mp^+\I$ if and only if
\begin{equation} \label{9.200}
E^*(p,q) : =\sup _{t\in (0,\i)}\| w \|_{{p},(0,t)} \|u\|_{q,(0,t)}^{-1} <\infty .
\end{equation}
The best possible constant in \eqref{9.100}, that is,
\begin{equation*}
F^*(p,q) : = \sup _{g \in \mp^+\I}\| g \|_{p,w,\I} / \|S^*g\|_{q,u,\I}
\end{equation*}
satisfies $F^*(p,q)\approx E^*(p,q)$.

{\rm (ii)} Let $0 < p < q \le +\i$ and $\frac{1}{r} = \frac{1}{p} - \frac{1}{q}$. Then \eqref{9.100}  holds if and only if
$$
E^*(p,q) : =\bigg( \int_{\I} \| w \|_{p,(0,t)}^{r}\,d\bigg(-\|u\|_{q,(0,t) }^{-r}\bigg) \bigg) ^{\frac{1}{r}} +\frac{\|w\|_{p,\I}} {\|u\|_{q,\I}}<\infty,
$$
and $F^*(p,q)\approx E^*(p,q)$.
\end{thm}


\section{Characterizations of $L_{p_1}(v_1) \hra \ces_{p_2,q}(u,v_2)$ and $L_{p_1}(v_1) \hra \cop_{p_2,q}(u,v_2)$}\label{s4}

In this section we characterize \eqref{emb1} and \eqref{emb2}.

The following theorem is true.
\begin{thm}\label{Emb-Lp-Ces}
Let $0 < p_2  \le p_1 \le \i$, $0 < q \le \i$, $v_1,\,v_2 \in \W\I$ and $u \in \O_q$.

{\rm (i)} If $p_1 \le q$, then
\begin{equation*}
\|\Id\|_{L_{p_1}(v_1) \rw \ces_{p_2,q}(u,v_2)} \ap \sup_{t \in \I} \big\| v_1^{-1} v_2 \big\|_{p_1 \rw p_2, (0,t)}\|u\|_{q,(t,\i)}
\end{equation*}
uniformly in $u \in \O_q$.

{\rm (ii)} If $q < p_1$, then
\begin{equation*}
\|\Id\|_{L_{p_1}(v_1) \rw \ces_{p_2,q}(u,v_2)} \ap
\bigg( \int_{(0,\infty)} \big\| v_1^{-1} v_2 \big\|_{p_1 \rw p_2,(0,t)}^{p_1 \rw q} \,d \bigg( - \|u\|_{q,(t,\i)}^{p_1 \rw q}\bigg) \bigg)^{\frac{1}{p_1 \rw q}}
\end{equation*}
uniformly in $u \in \O_q$.
\end{thm}
\begin{proof}
Let $p_2 < \infty$. Since
\begin{align*}
\|\Id\|_{L_{p_1}(v_1) \rw \ces_{p_2,q}(u,v_2)}  = \sup_{f \in \M^+ \I} \frac{\big\| \|f\|_{p_2,v_2,(0,\cdot)}\big\|_{q,u,(0,\infty)}}{\|f\|_{p_1,v_1,\I}} = \left(\sup_{g \in \M^+ \I} \frac{\| H (|g|)\|_{\frac{q}{p_2},u^{p_2},\I}}{\|g\|_{\frac{p_1}{p_2},[v_1v_2^{-1}]^{p_2},\I}}\right)^{\frac{1}{p_2}},
\end{align*}
it remains to apply Theorem \ref{HardyIneq}.

If $p_2 = \infty$, then
\begin{align*}
\|\Id\|_{L_{\infty}(v_1) \rw \ces_{\i,q}(u,v_2)}  = \sup_{f \in \M^+ \I} \frac{\big\| \|f\|_{\i,v_2,(0,\cdot)}\big\|_{q,u,(0,\infty)}}{\|f\|_{\i,v_1,\I}} = \sup_{g \in \M^+ \I}  \frac{\| (S(|g|))u \|_{q,\I} }{\big\|gv_1v_2^{-1}\big\|_{\i,\I}},
\end{align*}
and the statement follows by Theorem \ref{HardyforSupremal}.
\end{proof}

The following statement can be proved analogously.
\begin{thm}\label{Emb-Lp-Cop}
Let $0 < p_2  \le p_1 \le \i$, $0 < q \le \i$, $v_1,\,v_2 \in \W\I$ and $u \in \dual{\O}_q$.

{\rm (i)} If $p_1 \le q$, then
\begin{equation*}
\|\Id\|_{L_{p_1}(v_1) \rw \cop_{p_2,q}(u,v_2)} \ap \sup_{t \in \I} \big\| v_1^{-1} v_2 \big\|_{p_1 \rw p_2, (t,\infty)}\|u\|_{q,(0,t)}
\end{equation*}
uniformly in $u \in \dual{\O}_q$.

{\rm (ii)} If $q < p_1$, then
\begin{equation*}
\|\Id\|_{L_{p_1}(v_1) \rw \cop_{p_2,q}(u,v_2)} \ap
\bigg( \int_{(0,\infty)} \big\| v_1^{-1} v_2 \big\|_{p_1 \rw p_2,(t,\infty)}^{p_1 \rw q} \,d \bigg( \|u\|_{q,(0,t)}^{p_1 \rw q}\bigg) \bigg)^{\frac{1}{p_1 \rw q}}
\end{equation*}
uniformly in $u \in \dual{\O}_q$.
\end{thm}


\section{Characterizations of $ \ces_{p_2,q}(u,v_2) \hra L_{p_1}(v_1) $ and $  \cop_{p_2,q}(u,v_2)\hra L_{p_1}(v_1)$}\label{s5}

In this section we characterize the embeddings \eqref{emb3} and \eqref{emb4}.

\begin{thm}\label{Emb-Ces-Lp}
Let $0 < p_1 \le p_2 \le \infty$, $0 < q \le \infty$, $v_1,\,v_2 \in \W\I$ and $u \in \O_{q}$ .

{\rm (i)} If $q \le p_1$, then
\begin{equation*}
\|\Id\|_{\ces_{p_2,q}(u,v_2) \rw L_{p_1}(v_1)} \ap \sup_{t \in \I} \big\|v_1  v_2^{-1} \big\|_{p_1 \rw p_2,(t,\i)}\|u\|_{q,(t,\infty)}^{-1}
\end{equation*}
uniformly in $u \in \O_q$.

{\rm (ii)} If $p_1 < q$, then
\begin{align*}
\|\Id\|_{\ces_{p_2,q}(u,v_2) \rw L_{p_1}(v_1)} \ap & \bigg(
\int_{\I} \big\| v_1 v_2^{-1} \big\|_{p_1 \rw p_2,(t,\i)}^{q \rw
p_1} d \bigg( \|u\|_{q,(t-,\infty)}^{- q \rw p_1} \bigg)
\bigg)^{\frac{1}{q \rw p_1}}  +  \frac{\big\|v_1 v_2^{-1}\big\|_{p_1 \rw p_2,\I}}{\|u\|_{q,(0,\infty)}}
\end{align*}
uniformly in $u \in \O_q$.
\end{thm}

\begin{proof}
Let $p_2 < \infty$. Since
\begin{align*}
\|\Id\|_{\ces_{p_2,q}(u,v_2) \rw L_{p_1}(v_1)} = \sup_{f \in \M^+ \I} \frac{\|f\|_{p_1,v_1,\I}} {\big\| \|f\|_{p_2,v_2,(0,t)} \big\|_{q,u,(0,\infty)}} = \left(\sup_{g \in \M^+ \I} \frac{\big\|g ( v_1 v_2^{-1})^{p_2} \big\|_{\frac{p_1}{p_2},\I}}{\big\| H(|g|) u^{p_2} \big\|_{\frac{q}{p_2},\I}}\right)^{\frac{1}{p_2}},
\end{align*}
it remains to apply Theorem \ref{ReverseHardyIneq}.

If $p_2 = \infty$, then
\begin{align*}
\|\Id\|_{\ces_{p_2,q}(u,v_2) \rw L_{p_1}(v_1)} = \sup_{f \in \M^+ \I} \frac{\|f\|_{p_1,v_1,\I}} {\big\| \|f\|_{p_2,v_2,(0,t)} \big\|_{q,u,(0,\infty)}} = \sup_{g \in \M^+ \I} \frac{\big\|g v_1 v_2^{-1} \big\|_{p_1,\I}}{\big\| S(|g|) u \big\|_{q,\I}},
\end{align*}
and the statement follows by Theorem \ref{ReverseHardySupremal}.
\end{proof}

The following statement can be proved analogously.
\begin{thm}\label{Emb-Cop-Lp}
Let $0 < p_1 \le p_2 \le \infty$, $0 < q \le \infty$, $v_1,\,v_2 \in \W\I$ and $u \in \dual\O_q$.

{\rm (i)} If $q \le p_1$, then
\begin{equation*}
\|\Id\|_{\cop_{p_2,q}(u,v_2) \rw L_{p_1}(v_1)} \ap \sup_{t \in \I} \big\|v_1  v_2^{-1} \big\|_{p_1 \rw p_2,(0,t)}\|u\|_{q,(0,t)}^{-1}
\end{equation*}
uniformly in $u \in \dual\O_q$.

{\rm (ii)} If $p_1 < q$ then
\begin{align*}
\|\Id\|_{\cop_{p_2,q}(u,v_2) \rw L_{p_1}(v_1)} \ap & \bigg(
\int_{\I} \big\| v_1 v_2^{-1} \big\|_{p_1 \rw p_2,(0,t)}^{q \rw
p_1} d \bigg( - \|u\|_{q,(0,t+)}^{- q \rw p_1} \bigg)
\bigg)^{\frac{1}{q
\rw p_1}} +  \frac{\big\|v_1 v_2^{-1}\big\|_{p_1 \rw p_2,\I}}{\|u\|_{q,(0,\infty)}}
\end{align*}
uniformly in $u \in \dual\O_q$.
\end{thm}

\begin{defi}
Let $X$ be a set of functions from $\M\I$, endowed with a positively
homogeneous functional $\|\cdot \|_X$, defined for every $f\in \M\I$
and such that $f\in X$ if and only if $\|f\|_X<\i$. We define the
associate space $X'$ of $X$ as the set of all functions $f\in \M\I$
such that $\|f\|_{X'}<\i$, where
$$
\|f\|_{X'}=\sup \bigg\{\int_{\I}|f(x)g(x)|\,dx : \,\,\|g\|_X \leq
1\bigg\}.
$$
\end{defi}

In particular, Theorems \ref{IterHar3.2.0} and \ref{Emb-Cop-Lp}
allows us to give a characterization of the associate spaces of
weighted Ces\`{a}ro and Copson function spaces.
\begin{thm}\label{thm.6.6}
Assume $1\leq p<\i$, $0 < q \leq \i$. Let  $u \in \O_q$ and $v \in
\W\I$. Set
$$
X=\ces_{p,q}(u,v).
$$

{\rm (i)} Let $0 < q \leq 1$. Then
$$
\|f\|_{X'}\ap \sup_{t \in \I} \|f \|_{p',v^{- 1},(t,\i)} \|u\|_{q,(t,\i)}^{-1},
$$
with the positive constants in equivalence independent of $f$.

{\rm (ii)} Let $1 < q \leq \i$. Then
$$
\|f\|_{X'}\ap \bigg(\int_{\I} \|f \|_{{p'},v^{-1},(t,\i)}^{q'}d
\bigg(\|u\|_{q,(t-,\i)}^{-q'}\bigg)\bigg)^{\frac{1}{q'}}+\|f\|_{p',v^{-1},\I}\|u\|_{q,\I}^{-1},
$$
with the positive constants in equivalence independent of $f$.
\end{thm}

\begin{thm}\label{thm.6.5}
Assume $1 \leq p<\i$, $0 < q \leq \i$. Let  $u \in \dual{\O}_q$ and
$v \in \W\I$. Set
$$
X=\cop_{p,q}(u,v).
$$

{\rm (i)} Let $0 < q \leq 1$. Then
$$
\|f\|_{X'}\ap \sup_{t \in \I} \|f \|_{p',v^{- 1},(0,t)} \|u\|_{q,(0,t)}^{-1},
$$
with the positive constants in equivalence independent of $f$.

{\rm (ii)} Let $1 < q \leq \i$. Then
$$
\|f\|_{X'}\ap \bigg(\int_{\I} \|f \|_{p',v^{-1},(0,t)}^{q'} d
\bigg(- \|u\|_{q,(0,t+)}^{-q'} \bigg) \bigg)^{\frac{1}{q'}} + \|f\|_{p',v^{-1},\I} \|u\|_{q,\I}^{-1},
$$
with the positive constants in equivalence independent of $f$.
\end{thm}


\section{The iterated Hardy-type inequalities}\label{Iterated Hardy Type Inequalities}\label{s6}

In this section we recall characterizations of weighted iterated Hardy-type inequalities
\begin{equation}\label{eq.4.11}
\big\|\|H^* f\|_{p,u,(0,\cdot)}\big\|_{q,w,\I}\leq c \,\|f\|_{\theta,v,\I},~f \in \M^+\I,
\end{equation}
where $0 < p,\,q \le \infty$, $1 < \theta < \infty$.

Note that weighted iterated Hardy-type inequalities
have been intensely investigated recently (see, for instance, \cite{gmp} and \cite{GogMusPers2}, when $0 < p < \infty$, $0 < q \le \infty$, $1 \le \theta \le \infty$, and  \cite{gop}, when $p = \infty$, $0 < q < \infty$, $1 \le \theta < \infty$. For more detailed  information see recent papers \cite{GogMusIHI} and \cite{GogMusISI}). There exists different solutions of these inequalities. We will use the characterizations from  \cite{gmp} and \cite{gop}.

Everywhere in this section, $u$, $v$ and $w$ are weights on $\I$, and we denote
$$
U(t)=\int_0^t u(\tau)\,d\tau \qquad \mbox{and}\qquad V_{\t}(t)=
\int_t^{\i}{v(\tau)}^{1-\t'}d\tau  \quad \mbox{for}\,\, 1<\t<\i.
$$
We assume that $u$ is such that $U(t)>0$ for every $t\in\I$.

\begin{defi}
    Let $U$ be a continuous, strictly increasing function on $[0,\i)$ such that $U(0)=0$ and $\lim_{t\rw\i}U(t)=\i$. Then we say that $U$ is admissible.
\end{defi}

Let $U$ be an admissible function. We say that a function $\vp$ is $U$-quasiconcave if $\vp$ is equivalent to an increasing function on $(0,\i)$ and ${\vp} / {U}$ is equivalent to a decreasing function on $(0,\infty)$. We say that a $U$-quasiconcave function $\vp$ is non-degenerate if
$$
\lim_{t\rw 0+} \vp(t) = \lim_{t\rw \i} \frac{1}{\vp(t)} = \lim_{t\rw \i} \frac{\vp(t)}{U(t)} = \lim_{t\rw 0+} \frac{U(t)}{\vp(t)} =0.
$$
The family of non-degenerate $U$-quasiconcave functions is denoted by $Q_U$.
\begin{defi}
    Let $U$ be an admissible function, and let $w$ be a nonnegative measurable function on $(0,\i)$. We say that the function $\vp$ defined by
    \begin{equation*}
    \vp(t)=U(t)\int_0^{\infty} \frac{w(\tau)\,d\tau}{U(\tau)+U(t)}, \qq t\in (0,\i),
    \end{equation*}
    is a fundamental function of $w$ with respect to $U$. One will also say that $w(s)\,ds$ is a representation measure of $\vp$ with respect to $U$.
\end{defi}

Denote by
$$
{\mathcal U}(x,t): = \frac{U(x)}{U(t)+U(x)}.
$$

\begin{rem}\label{nondegrem}
    Let $\vp$ be the fundamental function of $w$ with
    respect to $U$.
    Assume that
    \begin{equation*}
    \int_0^{\infty}\frac{w(\tau)\,d\tau}{U(\tau)+U(t)}<\i, ~ t> 0, \qq \int_0^1 \frac{w(\tau)\,d\tau}{U(\tau)}=\int_1^{\infty}w(\tau)\,d\tau=\infty.
    \end{equation*}
    Then $\vp\in Q_{U}$.
\end{rem}

First we recall the characterization of \eqref{eq.4.1}, when $p < \infty$ and $q < \infty$.
\begin{thm}\cite[Theorem 3.1]{gmp}\label{IterHar3.1.0}
    Let $0<q<\i$, $0<p<\i$, $1<\t  < \i$ and let $u,\,v,\,w$ be
    weights. Assume that $u$ be such that $U$ is
    admissible and $\vp \in Q_{U^{\frac{q}{p}}}$, where $\vp$ is defined by
    \begin{equation}\label{eq.4.3}
    \vp(x) =\int_0^{\i}{\mathcal U}(x,\tau)^{\frac{q}{p}}w(\tau)\,d\tau
    \qquad\mbox{for all} \qquad x\in \I.
    \end{equation}
    Then inequality
    \begin{equation}\label{eq.4.1}
    \bigg(\int_0^{\i}\bigg(\frac{1}{U(t)}\int_0^t\bigg(\int_{\tau}^{\i}h(z)dz\bigg)^p u(\tau)\,d\tau\bigg)
    ^{\frac{q}{p}}w(t)dt\bigg)^{\frac{1}{q}}\leq c \,\bigg(\int_0^{\infty} h(t)^{\theta} v(t)\,dt\bigg)^{\frac{1}{\theta}},
    \end{equation}
    holds for every measurable function $f$ on $\I$ if and only if

    {\rm (i)} $\t \leq \min\{p,q\}$
    $$
    A_1:=\sup_{x\in (0,\i)}\bigg(\int_0^{\i}{\mathcal
        U}(x,\tau)^{\frac{q}{p}}w(\tau)\,d\tau\bigg)^{\frac{1}{q}}
    \sup_{t\in (0,\infty)}{\mathcal
        U}(t,x)^{\frac{1}{p}}{V_{\t}(t)}^{\frac{1}{\t'}}<\i.
    $$
    Moreover, the best constant $c$ in \eqref{eq.4.1} satisfies $c\ap
    A_1$.

    {\rm (ii)} $q<\t \le p$, $l=\frac{\t q}{\t-q}$
    \begin{gather*}
    A_2:=\bigg(\int_0^{\i} \bigg(\int_0^{\i}{\mathcal
        U}(x,\tau)^{\frac{q}{p}}w(\tau)\,d\tau\bigg)^{\frac{l-q}{q}}{w(x)} \sup_{t \in \I}
    {\mathcal U}(t,x)^{\frac{l}{p}}{V_{\t}(t)}^{\frac{l}{\t'}}dx\bigg)^{\frac{1}{l}}<\i.
    \end{gather*}
    Moreover, the best constant $c$ in \eqref{eq.4.1} satisfies $c\ap
    A_2$.

    {\rm (iii)} $p<\t \leq q$, $r=\frac{\t p}{\t-p}$
    \begin{gather*}
    A_3:= \sup_{x\in\I}\bigg(\int_0^{\i} {\mathcal
        U}(x,\tau)^{\frac{q}{p}}w(\tau)\,d\tau\bigg)^{\frac{1}{q}} \bigg(
    \int_0^{\i}{\mathcal U}(t,x)^{\frac{r}{p}}
    {V_{\t}(t)}^{\frac{r}{p'}}{v(t)}^{1-\t'}dt
    \bigg)^{\frac{1}{r}}<\i.
    \end{gather*}
    Moreover, the best constant $c$ in \eqref{eq.4.1} satisfies $c\ap
    A_3$.

    {\rm (iv)} $\max\{p,q\}<\t$, $r=\frac{\t p}{\t-p}$,
    $l=\frac{\t q}{\t-q}$
    \begin{gather*}
    A_4:=\bigg(\int_0^{\i}\bigg(\int_0^{\i}{\mathcal
        U}(x,\tau)^{\frac{q}{p}}w(\tau)\,d\tau\bigg)^{\frac{l-q}{q}}w(x) \bigg(
    \int_0^{\i}{\mathcal U}(t,x)^{\frac{r}{p}}
    {V_{\t}(t)}^{\frac{r}{p'}}{v(t)}^{1-\t'}dt
    \bigg)^{\frac{l}{r}}dx\bigg)^{\frac{1}{l}}<\i.
    \end{gather*}
    Moreover, the best constant $c$ in \eqref{eq.4.1} satisfies $c\ap
    A_4$.
\end{thm}

\begin{rem}\label{limsupcondition}
    Suppose that $\vp (x) < \i$ for all $x \in (0,\i)$, where $\vp$
    defined by
    \begin{equation*}
    \vp(x)=\esup_{t\in(0,x)}{U(t)}
    \esup_{\tau\in(t,\i)}\frac{w(\tau)}{U(\tau)},~~t\in\I.
    \end{equation*}
    If
    $$
    \limsup_{t \rightarrow 0 +} w(t) = \limsup_{t \rightarrow +\infty} \frac{1}{w(t)} = \limsup_{t \rightarrow 0 +} \frac{U(t)}{w(t)} = \limsup_{t \rightarrow +\infty} \frac{w(t)}{U(t)} = 0,
    $$
    then
    $\vp\in Q_{U}$.
\end{rem}
We now state the announced characterization of \eqref{eq.4.2}, when $p < \infty$ and $q = \infty$.
\begin{thm}\cite[Theorem 3.2]{gmp}\label{IterHar3.2.0}
    Let $0<p<\i$, $1<\t<\i$ and let $u,\,v,\,w$ be weights. Assume that $u$ is
    such that $U$ is admissible and $\vp \in Q_{U^{\frac{1}{p}}}$, where $\vp$ is defined by  \begin{equation}\label{eq.2.0}
    \vp(x)=\esup_{t\in\I}w(t){\mathcal U}(x,t)^{\frac{1}{p}},~~x\in\I.
    \end{equation}
    Then inequality
    \begin{equation}\label{eq.4.2}
    \esup_{t\in\I}w(t)\bigg(\frac{1}{U(t)}\int_0^t\bigg(\int_{\tau}^{\i}h(z)dz\bigg)^pu(\tau)\,d\tau\bigg)
    ^{\frac{1}{p}} \leq c \,\bigg(\int_0^{\infty} h(t)^{\theta} v(t)\,dt\bigg)^{\frac{1}{\theta}},
    \end{equation}
    holds for every measurable
    function $f$ on $\I$ if and only if

    {\rm (i)} $\t\leq p$ and
    \begin{gather*}
    B_1:=\sup_{x\in\I}\esup_{\tau\in\I}w(\tau){\mathcal
        U}(x,\tau)^{\frac{1}{p}} \sup_{t\in
        (0,\i)}{\mathcal
        U}(t,x)^{\frac{1}{p}} {V_{\t}(t)}^{\frac{1}{\t'}}<\i.
    \end{gather*}
    Moreover, the best constant $c$ in \eqref{eq.4.2} satisfies $c\ap
    B_1$.

    {\rm (ii)} $p<\t$, $r=\frac{\t p}{\t-p}$ and
    \begin{gather*}
    B_2:= \sup_{x\in\I}\esup_{\tau\in\I}w(\tau){\mathcal
        U}(x,\tau)^{\frac{1}{p}} \bigg(\int_0^{\i}{\mathcal
        U}(t,x)^{\frac{r}{p}}
    {V_{\t}(t)}^{\frac{r}{p'}}{v(t)}^{1-\t'}dt\bigg)^{\frac{1}{r}}<\i.
    \end{gather*}
    Moreover, the best constant $c$ in \eqref{eq.4.2} satisfies $c\ap
    B_2$.
\end{thm}

For a given weight $v$, $0 \le a < b \le \infty$ and $1 \le \theta < \infty$, we denote
$$
v_{\theta} (a,b) = \begin{cases}
\bigg ( \int\limits_a^b [v(t)]^{1-\theta'}dt\bigg)^{\frac{1}{p'}} & \qquad \mbox{when} ~ 1 < \theta < \infty, \\
\esup\limits_{t \in (a,b)} \, [v(t)]^{-1} & \qquad \mbox{when} ~ \theta = 1.
\end{cases}
$$

Finally, recall the characterization of \eqref{eq.4.2}, when $p = \infty$ and $q < \infty$.
\begin{thm}\cite[Theorems 4.1 and 4.4]{gop} \label{IterHarSupremal.1}
Let $1 \le \theta < \infty$, $0 < q < \infty$ and let $u \in \W\I \cap C\I$. Assume that $v,\,w \in \W\I$ be such that
$$
0 < \int_x^{\infty} v(\tau)\,d\tau < \i \qquad \mbox{and} \qquad  0 < \int_x^{\infty} w(\tau)\,d\tau < \infty \qq \mbox{for all} \qq x > 0.
$$
Then inequality
\begin{equation}\label{ISI.2}
\bigg(\int_0^{\infty} \bigg( \sup_{\tau \in (0,t)} \,u(\tau) \int_{\tau}^{\infty} h(z)\,dz\bigg)^q w(t)\,dt \bigg)^{\frac{1}{q}} \leq c \,\bigg(\int_0^{\infty} h(t)^{\theta} v(t)\,dt\bigg)^{\frac{1}{\theta}}
\end{equation}
is satisfied with the best constant $c$ if and only if:

{\rm (i)} $\theta \le q$, and in this case $c \ap A_1$, where
$$
A_1: = \sup_{x \in \I}\bigg( \bigg[\sup_{\tau \in (0,x)} u(\tau)\bigg]^q \int_x^{\infty} w(\tau)\,d\tau + \int_0^x  \bigg[\sup_{\tau \in (0,t)} u(\tau)\bigg]^q  w(t)\,dt\bigg)^{\frac{1}{q}}v_{\theta}(x,\infty);
$$

{\rm (ii)} $q < \theta$ and $\frac{1}{r} = \frac{1}{q} - \frac{1}{\theta}$,  and in this case $c \ap B_1 + B_2$, where
\begin{align*}
B_1: & = \bigg(\int_0^{\infty} \bigg(\int_0^x  \bigg[\sup_{\tau \in (0,t)} u(\tau)\bigg]^q w(t)\,dt\bigg)^{\frac{r}{\theta}} \bigg[\sup_{\tau \in (0,x)} u(\tau)\bigg]^q \bigg[v_{\theta}(x,\infty)\bigg]^r w(x)\,dx \bigg)^{\frac{1}{r}}, \\
B_2: & = \bigg(\int_0^{\infty} \bigg( \int_x^{\infty} w(\tau)\,d\tau \bigg)^{\frac{r}{\theta}} \bigg[\sup_{\tau \in (0,x)} \bigg[\sup_{y \in (0,\tau)} u(y)\bigg] v_{\theta} (\tau,\infty) \bigg]^r  w(x)\,dx \bigg)^{\frac{1}{r}}.
\end{align*}
\end{thm}

\section{Embeddings Between $\cop_{p_1,q_1}(u_1,v_1)$ and $\ces_{p_2,q_2}(u_2,v_2)$}\label{s7}

In this section we characterize the embeddings between weighted Copson and Ces\`{a}ro function spaces.

From now on, we will denote
$$
v(x) : = v_1(x)^{-1} v_2(x),
\quad
V(x) : = \| v\|_{p_1 \rw p_2,(0,x)},
\quad \mbox{and} \quad
{\mathcal V}(t,x):= \frac{V(t)}{V(t)+V(x)}, ~ (t > 0,\,x > 0).
$$

\begin{lem}\label{triviality}
Let $0<p_1, p_2, q_1, q_2\le \i$ and $p_1<p_2$. Assume that $v_1,
v_2\in \W\I$, $u_1\in \dual{\O_{q_1}}$ and $u_2\in \O_{q_2}$. Then
$\cop_{p_1,q_1}(u_1,v_1)\not \hookrightarrow
\ces_{p_2,q_2}(u_2,v_2)$.
\begin{proof}
Assume that $\cop_{p_1,q_1}(u_1,v_1) \hookrightarrow \ces_{p_2,q_2}(u_2,v_2)$ holds. Then there exist $c>0$ such that
\begin{equation*}
\|f\|_{\ces_{p_2,q_2}(u_2,v_2)} \leq c \,\|f\|_{\cop_{p_1,q_1}(u_1,v_1)}
\end{equation*}
holds for all $f \in \M^+\I$. Let $\tau\in \I$ and $f=0$ on $(\tau,\i)$. Thus, we have
\begin{align}
\|f\|_{\ces_{p_2,q_2}(u_2,v_2)}&= \big\| \|f\|_{p_2,v_2,(0,\cdot)} \big\|_{q_2,u_2,\I}\notag\\
&\geq \big\| \|f\|_{p_2,v_2,(0,\cdot)} \big\|_{q_2,u_2,(\tau,\i)}\notag\\
&\geq \| u_2 \|_{q_2,(\tau,\i)} \|f\|_{p_2,v_2,(0,\tau)} \label{1}
\end{align}
and
\begin{align}
\|f\|_{\cop_{p_1,q_1}(u_1,v_1)}&= \big\| \|f\|_{p_1,v_1,(\cdot,\i)} \big\|_{q_1,u_1,\I}\notag\\
&\leq \big\| \|f\|_{p_1,v_1,(\cdot,\i)} \big\|_{q_1,u_1,(0,\tau)}\notag\\
&\leq \| u_1 \|_{q_1,(0,\tau)} \|f\|_{p_1,v_1,(0,\tau)} \label{2}.
\end{align}
Combining \eqref{1} with \eqref{2}, we can assert that
\begin{equation*}
\| u_2 \|_{q_2,(\tau,\i)} \|f\|_{p_2,v_2,(0,\tau)}\leq c  \| u_1 \|_{q_1,(0,\tau)} \|f\|_{p_1,v_1,(0,\tau)}.
\end{equation*}
Since $u_1\in \dual{\O_{q_1}}$ and $u_2\in \O_{q_2}$, we conclude that $L_{p_1}(v_1)\hookrightarrow L_{p_2}(v_2)$, which is a contradiction.
\end{proof}
\end{lem}

\begin{thm}
Let $0<p_1=q_1<\i$, $0<p_2=q_2<\i$, $v_1, v_2\in \W\I$, $u_1\in \dual{\O_{q_1}}$ and $u_2 \in \O_{q_2}$.
Then
\begin{equation*}
\|\Id\|_{\cop_{p_1,q_1}(u_1,v_1)\rw \ces_{p_2,q_2}(u_2,v_2)}\ap \bigg\| \|u_1\|_{p_1,(0,\cdot)}^{- 1} \|u_2\|_{p_2,(\cdot,\infty)} \bigg\|_{p_1 \rw p_2,v,\I}.
\end{equation*}
\end{thm}

\begin{proof}
In view of Lemmas \ref{Cespp} and \ref{Coppp}, we have that
$$
\|\Id\|_{\cop_{p_1,q_1}(u_1,v_1)\rw \ces_{p_2,q_2}(u_2,v_2)} = \|\Id\|_{L_{p_1}(w_1) \rw L_{p_2}(w_2)}
$$
with $w_1(x) = v_1(x)\|u_1\|_{p_1,(0,x)}$ and $w_2(x) = v_2(x)\|u_2\|_{p_2,(x,\infty)}$, $x>0$.
\end{proof}

\begin{thm}
Let $0<p_1,p_2,q_1,q_2<\i$, $p_1=q_1$ and $p_2\neq q_2$. Let  $v_1, v_2\in \W\I$, $u_1\in \dual{\O_{q_1}}$ and $u_2 \in \O_{q_2}$.

{\rm (i)} If $p_1\le q_2$, then
\begin{equation*}
\|\Id\|_{\cop_{p_1,q_1}(u_1,v_1)\rw \ces_{p_2,q_2}(u_2,v_2)} \ap \sup_{t \in \I} \big\| \|u_1\|_{p_1,(0,\cdot)}^{-1} \big\|_{p_1 \rw p_2,v,(0,t)}\|u_2\|_{q_2,(t,\infty)},
\end{equation*}

{\rm (ii)}  If $q_2 < p_1$, then
\begin{align*}
\|\Id\|_{\cop_{p_1,q_1}(u_1,v_1)\rw \ces_{p_2,q_2}(u_2,v_2)} \ap \bigg(\int_{(0,\infty)} \big\| \|u_1\|_{p_1,(0,\cdot)}^{-1} \big\|_{p_1 \rw p_2,v,(0,t)}^{p_1 \rw q_2} d \,\bigg(- \|u_2\|_{q_2,(t,\infty)}^{p_1 \rw q_2} \bigg)\bigg)^{\frac{1}{p_1 \rw q_2}}.
\end{align*}
\end{thm}
\begin{proof}
In view of Lemma~\ref{Coppp}, we have that
$$
\|\Id\|_{\cop_{p_1,q_1}(u_1,v_1)\rw \ces_{p_2,q_2}(u_2,v_2)} = \|\Id\|_{L_{p_1}(w_1) \rw \ces_{p_2,q_2}(u_2,v_2)}
$$
with $w_1(x) = v_1(x)\|u_1\|_{p_1,(0,x)}$, $x>0$. Then the result follows from Theorem~\ref{Emb-Lp-Ces}.
\end{proof}

\begin{thm}
Let $0<p_1,p_2,q_1,q_2<\i$, $p_1 \neq q_1$ and $p_2= q_2$. Let  $v_1, v_2\in \W\I$, $u_1\in \dual{\O_{q_1}}$ and $u_2 \in \O_{q_2}$.

{\rm (i)} If $q_1 \le p_2$, then
\begin{equation*}
\|\Id\|_{\cop_{p_1,q_1}(u_1,v_1)\rw \ces_{p_2,q_2}(u_2,v_2)} \ap \sup_{t \in \I}  \|u_1\|_{q_1,(0,t)}^{-1} \,\big\| \|u_2\|_{p_2,(\cdot,\infty)} \big\|_{p_1 \rw p_2},
\end{equation*}

{\rm (ii)} If $p_2<q_1$ then
\begin{align*}
\|\Id\|_{\cop_{p_1,q_1}(u_1,v_1)\rw \ces_{p_2,q_2}(u_2,v_2)} \ap & \, \bigg(\int_0^\i \big\| \|u_2\|_{p_2,(\cdot,\infty)}\big\|_{p_1 \rw p_2,v,(0,t)}^{q_1 \rw p_2}  d \,\bigg( -\|u_1\|_{q_1,(0,t)}^{- q_1 \rw p_2}\bigg)\bigg)^{\frac{1}{q_1 \rw p_2}} \\
& + \|u_1\|_{q_1,(0,\infty)}^{-1} \,\big\| \|u_2\|_{p_2,(\cdot,\infty)}\big\|_{p_1 \rw p_2,v,(0,\infty)}.
\end{align*}
\end{thm}
\begin{proof}
In view of Lemma~\ref{Cespp}, we have that
$$
\|\Id\|_{\cop_{p_1,q_1}(u_1,v_1)\rw \ces_{p_2,q_2}(u_2,v_2)} =
\|\Id\|_{\cop_{p_1,q_1}(u_1,v_1)\rw L_{p_2}(w_2)}
$$
with $w_2(x) = v_2(x)\|u_2\|_{p_2,(x,\infty)}$, $x>0$. Then the result follows from Theorem~\ref{Emb-Cop-Lp}.
\end{proof}

The following lemma is true.
\begin{lem}\label{mainlemma}
Let $0<p_1, p_2, q_1, q_2<\i$, $p_2\le p_1$ and $p_2<q_2$. Let $v_1, v_2\in \W\I$, $u_1 \in \dual{\O_{q_1}}$ and $u_2\in \O_{q_2}$.
Then
$$
\|\Id\|_{\cop_{p_1,q_1}(u_1,v_1)\rw \ces_{p_2,q_2}(u_2,v_2)} = \left\{\sup_{g\in \M^+\I} \ddfrac{\big\|\Id\big\|_{\cop_{p_1,q_1}(u_1,v_1)\rw L_{p_2}\big(v_2 H^*(g)^{\frac{1}{p_2}}\big)}^{p_2}}{\|g\|_{\frac{q_2}{q_2-p_2},u_2^{-p_2},\I}} \right\}^{\frac{1}{p_2}}.
$$
\end{lem}
\begin{proof}
By duality, interchanging suprema, we have that
\begin{align*}
\|\Id\|_{\cop_{p_1,q_1}(u_1,v_1)\rw \ces_{p_2,q_2}(u_2,v_2)} & \\
& \hspace{-3cm} =\sup_{f\in \M^+\I} \ddfrac{\|f\|_{\ces_{p_2,q_2}(u_2,v_2)}}{\|f\|_{\cop_{p_1,q_1}(u_1,v_1)}} \\
& \hspace{-3cm} = \sup_{f\in \M^+\I} \ddfrac{1}{\|f\|_{\cop_{p_1,q_1}(u_1,v_1)}} \sup_{g\in \M^+\I} \ddfrac{\bigg(\int_0^\i \bigg(\int_0^{\tau} f(x)^{p_2}v_2(x)^{p_2}\,dx\bigg)\,g(\tau)\,d\tau\bigg)^{\frac{1}{p_2}}} {\|g\|_{\frac{q_2}{q_2-p_2},u_2^{-p_2},\I}^{\frac{1}{p_2}}}\\
& \hspace{-3cm} = \sup_{g\in \M^+\I} \ddfrac{1}{\|g\|_{\frac{q_2}{q_2-p_2},u_2^{-p_2},\I}^{\frac{1}{p_2}}} \sup_{f\in \M^+\I} \ddfrac{\bigg(\int_0^\i \bigg(\int_0^{\tau} f(x)^{p_2}v_2(x)^{p_2}\,dx\bigg)\,g(\tau)\,d\tau\bigg)^{\frac{1}{p_2}}} {\|f\|_{\cop_{p_1,q_1}(u_1,v_1)}}.
\end{align*}
Applying the Fubini's Theorem, we get that
\begin{align}
\|\Id\|_{\cop_{p_1,q_1}(u_1,v_1)\rw \ces_{p_2,q_2}(u_2,v_2)} & \notag \\
& \hspace{-3cm} = \sup_{g\in \M^+\I} \ddfrac{1}{\|g\|_{\frac{q_2}{q_2-p_2},u_2^{-p_2},\I}^{\frac{1}{p_2}}} \sup_{f\in \M^+\I} \ddfrac{\bigg(\int_0^\i f(x)^{p_2}v_2(x)^{p_2} \bigg(\int_x^\i g(\tau)\,d\tau\bigg) dx \bigg)^{\frac{1}{p_2}}} {\|f\|_{\cop_{p_1,q_1}(u_1,v_1)}}\notag\\
& \hspace{-3cm} = \sup_{g\in \M^+\I} \ddfrac{1}{\|g\|_{\frac{q_2}{q_2-p_2},u_2^{-p_2},\I}^{\frac{1}{p_2}}} \,\|\Id\|_{\cop_{p_1,q_1}(u_1,v_1)\rw L_{p_2}\big(v_2 (\cdot) H^*(g)^{\frac{1}{p_2}}\big)}.\label{GommeninIlkDenkligi}
\end{align}
\end{proof}

\begin{thm}\label{maintheorem1}
Let $0<p_1, p_2, q_1, q_2<\i$, $p_2 < p_1$, $q_1 \le p_2 < q_2$. Let $v_1, v_2\in \W\I$, $u_1 \in \dual{\O_{q_1}}$ and $u_2\in \O_{q_2}$.
Assume that $V$ is admissible and
$$
\vp_1(x):= \esup_{t\in \I} V(t){\mathcal V}(x,t)
\|u_1\|_{q_1,(0,t)}^{-1} \in Q_{V^{\frac{1}{p_1 \rw p_2}}}.
$$

{\rm (i)} If $p_1\le q_2$, then
\begin{equation*}
\|\Id\|_{\cop_{p_1,q_1}(u_1,v_1)\rw \ces_{p_2,q_2}(u_2,v_2)}\ap \sup_{x\in \I} \vp_1(x) \sup_{t\in \I} {\mathcal V}(t,x) \|u_2\|_{q_2,(t,\infty)}.
\end{equation*}

{\rm (ii)} If $q_2<p_1$, then
\begin{equation*}
\|\Id\|_{\cop_{p_1,q_1}(u_1,v_1)\rw \ces_{p_2,q_2}(u_2,v_2)}\ap
\sup_{x\in \I} \vp_1(x) \bigg(\int_0^\i {\mathcal V}(t,x)^{p_1 \rw
q_2}  d\bigg( - \|u_2\|_{q_2,(t,\infty)}^{p_1 \rw q_2}\bigg)
\bigg)^{\frac{1}{p_1 \rw q_2}}.
\end{equation*}
\end{thm}

\begin{proof}
By Lemma \ref{mainlemma}, we have that
\begin{align*}
\|\Id\|_{\cop_{p_1,q_1}(u_1,v_1)\rw \ces_{p_2,q_2}(u_2,v_2)}  = \sup_{g\in \M^+\I} \frac{1}{\|g\|_{\frac{q_2}{q_2-p_2},u_2^{-p_2},\I}^{\frac{1}{p_2}}} \|\Id\|_{\cop_{p_1,q_1}(u_1,v_1)\rw L_{p_2}(v_2 H^*(g)^{\frac{1}{p_2}})}.
\end{align*}
Since $q_1\le p_2$, applying Theorem~[\ref{Emb-Cop-Lp}, (i)], we obtain that
\begin{align*}
\|\Id\|_{\cop_{p_1,q_1}(u_1,v_1)\rw \ces_{p_2,q_2}(u_2,v_2)} \ap \left\{ \sup_{g\in \M^+\I} \ddfrac {\sup_{t\in \I} \|u_1\|_{q_1,(0,t)}^{- p_2} \|H^*g\|_{\frac{p_1}{p_1 - p_2},v^{p_2},(0,t)}} {\|g\|_{\frac{q_2}{q_2-p_2},u_2^{-p_2},\I}}\right\}^{\frac{1}{p_2}}.
\end{align*}

\rm{(i)} If $p_1\le q_2$, then applying Theorem~[\ref{IterHar3.2.0}, (i)], we arrive at
\begin{equation*}
\|\Id\|_{\cop_{p_1,q_1}(u_1,v_1)\rw \ces_{p_2,q_2}(u_2,v_2)}\ap \sup_{x\in \I} \vp_1(x) \sup_{t\in \I} V(t,x) \|u_2\|_{q_2,(t,\infty)}.
\end{equation*}

\rm{(ii)} If $q_2 < p_1$, then applying Theorem~[\ref{IterHar3.2.0}, (ii)], we arrive at
\begin{equation*}
\|\Id\|_{\cop_{p_1,q_1}(u_1,v_1)\rw \ces_{p_2,q_2}(u_2,v_2)}\ap
\sup_{x\in \I} \vp_1(x) \bigg(\int_0^\i {\mathcal V}(t,x)^{p_1 \rw
q_2}  d\bigg( - \|u_2\|_{q_2,(t,\infty)}^{p_1 \rw q_2}\bigg)
\bigg)^{\frac{1}{p_1 \rw q_2}}.
\end{equation*}

\end{proof}

\begin{rem}
In view of Remark~\ref{limsupcondition}, if
\begin{align*}
\limsup_{t\rw 0+} V(t) \|u_1\|_{q_1,(0,t)}^{-1} = \limsup_{t\rw +\infty} V(t)\|u_1\|_{q_1,(0,t)} = \limsup_{t\rw 0+}\|u_1\|_{q_1,(0,t)} = \limsup_{t\rw +\infty} \|u_1\|_{q_1,(0,t)}^{-1} = 0,
\end{align*}
then $\vp_1 \in Q_{V^{\frac{1}{p_1 \rw p_2}}}$.
\end{rem}

\begin{thm}\label{maintheorem3}
Let $0<p_1, p_2, q_1, q_2<\i$, $p_2 < p_1$ and $p_2 < \min\{q_1,q_2\}$. Let $v_1, v_2\in \W\I$, $u_1 \in \dual{\O_{q_1}}$ and $u_2\in \O_{q_2}$.
Assume that $V$ is admissible and
$$
\vp_2(x):= \bigg(\int_0^\i [{\mathcal V}(x,t)V(t)]^{q_1 \rw p_2} d \bigg( - \|u_1\|_{q_1,(0,t)}^{- q_1 \rw p_2}\bigg)\bigg)^{\frac{1}{q_1 \rw p_2}} \in Q_{V^{\frac{1}{p_1 \rw p_2}}}.
$$

{\rm (i)} If $\max\{p_1,q_1\} \leq q_2$, then
\begin{align*}
\|\Id\|_{\cop_{p_1,q_1}(u_1,v_1)\rw \ces_{p_2,q_2}(u_2,v_2)} \ap & \sup_{x\in \I} \vp_2(x) \sup_{t\in \I} {\mathcal V}(t,x) \|u_2\|_{q_2,(t,\infty)}\\
& + \|u_1\|_{q_1,\I}^{-1} \sup_{t\in \I} V(t) \|u_2\|_{q_2,(t,\infty)}.
\end{align*}

{\rm (ii)} If $p_1\le q_2<q_1$, then
\begin{align*}
\|\Id\|_{\cop_{p_1,q_1}(u_1,v_1) \rw \ces_{p_2,q_2}(u_2,v_2)} & \\
& \hspace{-3cm} \ap  \bigg( \int_0^\i \vp_2(x)^{\frac{q_1 \rw q_2 \cdot q_1 \rw p_2}{q_2 \rw p_2}} V(x)^{q_1 \rw p_2}  \bigg(\sup_{t\in(0,\infty)} {\mathcal V}(t,x) \|u_2\|_{q_2,(t,\infty)}\bigg)^{q_1 \rw q_2} d \bigg( - \|u_1\|_{q_1,(0,x)}^{- q_1 \rw p_2}\bigg) \bigg)^{\frac{1}{q_1 \rw q_2}} \\
& \hspace{-2.5cm} + \|u_1\|_{q_1,\I}^{-1} \sup_{t\in \I} V(t) \|u_2\|_{q_2,(t,\infty)}.
\end{align*}

{\rm (iii)} If $q_1\le q_2<p_1$, then
\begin{align*}
\|\Id\|_{\cop_{p_1,q_1}(u_1,v_1) \rw \ces_{p_2,q_2}(u_2,v_2)} \ap & \sup_{x\in\I} \vp_2(x) \bigg(\int_0^\i {\mathcal V}(t,x)^{p_1 \rw q_2} d \bigg( - \|u_2\|_{q_2, (t,\infty)}^{p_1 \rw q_2}\bigg)\bigg)^{\frac{1}{p_1 \rw q_2}}\\
& + \|u_1\|_{q_1,\I}^{-1} \bigg( \int_0^\i V(t)^{p_1 \rw q_2} d \bigg( - \|u_2\|_{q_2, (t,\infty)}^{p_1 \rw q_2}\bigg) \bigg)^{\frac{1}{p_1 \rw q_2}}.
\end{align*}

{\rm (iv)} If $q_2<\min\{p_1,q_1\}$, then
\begin{align*}
\|\Id\|_{\cop_{p_1,q_1}(u_1,v_1) \rw \ces_{p_2,q_2}(u_2,v_2)} & \\
& \hspace{-4cm} \ap \bigg( \int_0^\i \vp_2(x)^{\frac{q_1 \rw q_2 \cdot q_1 \rw p_2}{q_2 \rw p_2}} V(x)^{q_1 \rw p_2}  \bigg(\int_0^{\infty}{\mathcal V}(t,x)^{p_1 \rw q_2} d \bigg(-\|u_2\|_{q_2,(t,\infty)}^{p_1 \rw q_2} \bigg) \bigg)^{\frac{q_1 \rw q_2}{p_1 \rw q_2}} d\bigg(-\|u_1\|_{q_1,(0,x)}^{- q_1 \rw p_2} \bigg) \bigg)^{\frac{1}{q_1 \rw q_2}}
\\
& \hspace{-3.5cm} +\|u_1\|_{q_1,\I}^{-1} \bigg( \int_0^\i V(t)^{p_1 \rw q_2} d \bigg( - \|u_2\|_{q_2, (t,\infty)}^{p_1 \rw q_2}\bigg) \bigg)^{\frac{1}{p_1 \rw q_2}}.
\end{align*}
\end{thm}

\begin{proof}
By  Lemma \ref{mainlemma}, applying Theorem~[\ref{Emb-Cop-Lp}, (ii)], we have that
\begin{align*}
\|\Id\|_{\cop_{p_1,q_1}(u_1,v_1) \rw \ces_{p_2,q_2}(u_2,v_2)} & \\
& \hspace{-3.5cm} \ap \|u_1\|_{q_1,\I}^{-1} \left\{\sup_{g\in \M^+\I} \ddfrac{ \|H^*g\|_{\frac{p_1}{p_1 - p_2},v^{p_2},(0,\infty)}} {\|g\|_{\frac{q_2}{q_2-p_2},u_2^{-p_2},\I}} \right\}^{\frac{1}{p_2}} \\
& \hspace{-3cm} + \left\{\sup_{g\in \M^+\I} \ddfrac{\bigg(\int_0^\i \|H^* g\|_{\frac{p_1}{p_1 - p_2},v^{p_2},(0,t)}^{\frac{q_1}{q_1-p_2}} d \bigg( - \|u_1\|_{q_1,(0,t)}^{- \frac{q_1 p_2}{q_1 - p_2}} \bigg)\bigg)^{\frac{q_1-p_2}{q_1}}}  {\|g\|_{\frac{q_2}{q_2-p_2},u_2^{-p_2},\I}} \right\}^{\frac{1}{p_2}}\\
& \hspace{-3.5cm} := C_1 + C_2.
\end{align*}

Note that
$$
C_1 = \|u_1\|_{q_1,\I}^{-1} \left\{
\|\Id\|_{L_{\frac{q_2}{q_2-p_2}}\big(u_2^{-p_2}\big) \rw \cop_{1,
\frac{p_1}{p_1-p_2}}\big(v^{p_2}, {\bf
1}\big)}\right\}^{\frac{1}{p_2}}
$$

Assume first that $p_1\leq q_2$.  Applying Theorem~[\ref{Emb-Lp-Cop}, (i)], we arrive at
\begin{align}\label{C_1 1}
C_1 \ap \|u_1\|_{q_1,\I}^{-1} \sup_{t\in \I} V(t)\, \|u_2\|_{q_2,(t,\infty)}.
\end{align}

{\rm (i)} Let  $q_1\le q_2$. Using Theorem~[\ref{IterHar3.1.0}, (i)], we obtain that
\begin{equation*}
C_2 \ap \sup_{x\in \I} \vp_2(x) \, \sup_{t\in (0,\infty)} {\mathcal V} (t,x) \,\|u_2\|_{q_2,(t,\infty)}.
 \end{equation*}
Consequently, the proof is completed in this case.

{\rm (ii)} Let $q_2<q_1$. Applying Theorem~[\ref{IterHar3.1.0}, (ii)], we have that
\begin{align*}
C_2 \ap \bigg( \int_0^\i \vp_2(x)^{\frac{q_1 \rw q_2 \cdot q_1 \rw p_2}{q_2 \rw p_2}} V(x)^{q_1 \rw p_2}  \bigg(\sup_{t\in(0,\infty)} {\mathcal V}(t,x) \|u_2\|_{q_2,(t,\infty)}\bigg)^{q_1 \rw q_2} d \bigg( - \|u_1\|_{q_1,(0,x)}^{- q_1 \rw p_2}\bigg) \bigg)^{\frac{1}{q_1 \rw q_2}},
\end{align*}
and the statement follows in this case.

Let us now assume that $q_2 < p_1$. Then using Theorem~[\ref{Emb-Lp-Cop}, (ii)], we have that
\begin{align}\label{C_1 2}
C_1\ap  \|u_1\|_{q_1,\I}^{-1} \bigg( \int_0^\i V(t)^{p_1 \rw q_2} d \bigg( - \|u_2\|_{q_2, (t,\infty)}^{p_1 \rw q_2}\bigg) \bigg)^{\frac{1}{p_1 \rw q_2}}.
\end{align}

{\rm (iii)} Let $q_1\le q_2$, then Theorem~[\ref{IterHar3.1.0}, (iii)] yields that
\begin{align*}
C_2\ap \sup_{x\in\I} \vp_2(x) \bigg(\int_0^\i {\mathcal V}(t,x)^{p_1 \rw q_2} d \bigg( - \|u_2\|_{q_2, (t,\infty)}^{p_1 \rw q_2}\bigg)\bigg)^{\frac{1}{p_1 \rw q_2}},
\end{align*}
these complete the proof in this case.

{\rm (iv)} If $q_2 < q_1$, then on using Theorem~[\ref{IterHar3.1.0}, (iv)], we arrive at
\begin{align*}
C_2 \ap \bigg( \int_0^\i \vp_2(x)^{\frac{q_1 \rw q_2 \cdot q_1 \rw p_2}{q_2 \rw p_2}} V(x)^{q_1 \rw p_2}  \bigg(\int_0^{\infty}{\mathcal V}(t,x)^{p_1 \rw q_2} d \bigg(-\|u_2\|_{q_2,(t,\infty)}^{p_1 \rw q_2} \bigg) \bigg)^{\frac{q_1 \rw q_2}{p_1 \rw q_2}} d\bigg(-\|u_1\|_{q_1,(0,x)}^{- q_1 \rw p_2} \bigg) \bigg)^{\frac{1}{q_1 \rw q_2}},
\end{align*}
and the proof follows.
\end{proof}

\begin{rem}
Assume that $\vp_2(x) < \infty, ~ x > 0$. In view of Remark~\ref{nondegrem}, if
$$
\int_0^1 \bigg(\int_0^t u_1^{q_1} \bigg)^{-\frac{q_1}{q_1-p_2}} u_1^{q_1}(t) \,dt = \int_1^{\infty} V(t)^{\frac{q_1 p_2}{q_1-p_2}} \bigg(\int_0^t u_1^{q_1} \bigg)^{-\frac{q_1}{q_1-p_2}} u_1^{q_1}(t) \,dt = \infty,
$$
then $\vp_2 \in Q_{V^{\frac{1}{p_1 \rw p_2}}}$.
\end{rem}

Now consider the case when $p_1 = p_2=p$.
\begin{thm}\label{maintheorem2}
Let $0<q_1< p < q_2<\i$. Assume that $v_1, v_2\in \W\I$, $u_1 \in \dual{\O_{q_1}}$ and $u_2\in \O_{q_2}$. Then
\begin{equation*}
\|\Id\|_{\cop_{p_1,q_1}(u_1,v_1)\rw \ces_{p_2,q_2}(u_2,v_2)}\ap \sup_{t \in \I} \big\| \|u_1\|_{q_1,(0,\cdot)}^{-1} \big\|_{\i,v,(0,t)} \|u_2\|_{q_2,(t,\i)}.
\end{equation*}
\end{thm}
\begin{proof}
By  Lemma \ref{mainlemma}, applying Theorem~[\ref{Emb-Cop-Lp}, (i)], we have that
\begin{align*}
\|\Id\|_{\cop_{p_1,q_1}(u_1,v_1)\rw \ces_{p_2,q_2}(u_2,v_2)} & \ap \left\{\sup_{g\in \M^+\I} \ddfrac{ \|H^*g\|_{\infty,v^p \|u_1\|_{q_1,(0,\cdot)}^{-p},\I}}
{\|g\|_{\frac{q_2}{q_2-p},u_2^{-p},\I}} \right\}^{\frac{1}{p}}\\
& = \left\{\|\Id\|_{L_{\frac{q_2}{q_2-p}}\big(u_2^{-p}\big)\rw\cop_{1,\i}\big(v^p \|u_1\|_{q_1,(0,\cdot)}^{-p}, {\bf 1}\big)}\right\}^{\frac{1}{p}}.
\end{align*}
Therefore, by Theorem~[\ref{Emb-Lp-Cop}, (i)],
\begin{equation*}
\|\Id\|_{\cop_{p_1,q_1}(u_1,v_1)\rw \ces_{p_2,q_2}(u_2,v_2)}\ap \sup_{t \in \I} \big\| \|u_1\|_{q_1,(0,\cdot)}^{-1} \big\|_{\i,v,(0,t)} \|u_2\|_{q_2,(t,\i)}.
\end{equation*}
\end{proof}

Before proceeding to the case $p = p_ 1 = p_2 < q_2$, we prove another variant of "gluing" lemma. The idea of proof comes in the same line as in \cite[Theorem 3.1]{gogperstepwall}.
\begin{lem}\label{gluing.lem}
    Let $\b$ be positive number and $u \in \W\I$, $g \in \M^+\I$. Assume that $h$ is a non-negative continuous function on $\I$. Then
    \begin{align*}
    \int_0^{\infty} \bigg( \int_0^{\infty} {\mathcal U}(x,t)g(t)\,dt \bigg)^{\b - 1} \bigg(\sup_{t\in \I} {\mathcal U}(t,x) h(t)\bigg)^{\b}g(x)\,dx & \\
    & \hspace{-5.5cm} \ap \int_0^{\infty} \bigg( \int_0^x g(t)\,dt\bigg)^{\b - 1} \bigg(\sup_{t \in (x, \infty)} h(t)\bigg)^{\b} g(x)\,dx \\
    & \hspace{-5cm} + \int_0^{\infty} \bigg( \int_x^{\infty} U(\tau)^{-1} g(\tau)\,d\tau\bigg)^{\b - 1}  \bigg( \sup_{t \in (0,x)} U(t)h(t)\bigg)^{\b}\,U(x)^{-1}g(x)\,dx.
    \end{align*}
\end{lem}

\begin{proof}
    Denote by
    \begin{align*}
    A_1 : & = \int_0^{\infty} \bigg( \int_0^x g(t)\,dt\bigg)^{\b - 1} \bigg(\sup_{t \in (x, \infty)} h(t)\bigg)^{\b} g(x)\,dx, \\
    A_2 : & = \int_0^{\infty} \bigg( \int_x^{\infty} U(\tau)^{-1} g(\tau)\,d\tau\bigg)^{\b - 1}  \bigg( \sup_{t \in (0,x)} U(t)h(t)\bigg)^{\b}\,U(x)^{-1}g(x)\,dx.
    \end{align*}
    Obviously,
    \begin{align*}
    \int_0^{\infty} \bigg( \int_0^{\infty} {\mathcal U}(x,t)g(t)\,dt \bigg)^{\b - 1} \bigg(\sup_{t\in \I} {\mathcal U}(t,x) h(t)\bigg)^{\b}\,g(x)dx  & \\
    & \hspace{-7cm} \ap \int_0^{\infty} \bigg( \int_0^x g(t)\,dt + U(x) \int_x^{\infty} U(t)^{-1}g(t) \,dt\bigg)^{\b - 1} \bigg(U(x)^{-1} \sup_{t \in (0,x)} U(t)h(t) + \sup_{t \in (x,\infty)}h(t)\bigg)^{\b} g(x)\,dx \\
    & \hspace{-7cm} \ap A_1 + A_2 + B_1 + B_2,
    \end{align*}
    where
    \begin{align*}
    B_1 : & = \int_0^{\infty} \bigg( \int_0^x g(t)\,dt\bigg)^{\b - 1} \bigg( U(x)^{-1} \sup_{t \in (0,x)} U(t)h(t)\bigg)^{\b} g(x)\,dx, \\
        B_2 : & = \int_0^{\infty} \bigg( \int_x^{\infty} U(t)^{-1}g(t) \, dt\bigg)^{\b - 1}  \bigg(U(x)\sup_{t \in (x,\infty)} h(t)\bigg)^{\b} U(x)^{- 1}g(x)\,dx.
    \end{align*}
    It is enough to show that $B_i \ls A_1 + A_2$, $i = 1,2$.

    Let us show that $B_1 \ls A_1 + A_2$. We will consider the case when $\int_0^{\infty} g(t)\,dt < \infty$ (The case when $\int_0^{\infty} g(t)\,dt = \infty$ is much simpler to treat). Define a sequence $\{x_k\}_{k = -\infty}^M$ such that $\int_0^{x_k} g(t)\,dt = 2^k$ if $-\infty < k \le M$ and $2^M \le \int_0^{\infty} g(t)\,dt < 2^{M+1}$. Then
    \begin{align*}
    B_1 & \le \int_0^{\infty} \bigg( \int_0^x g(t)\,dt\bigg)^{\b - 1} \bigg( \sup_{y \in (x,\infty)} U(y)^{-1} \sup_{t \in (0,y)} U(t)h(t)\bigg)^{\b} g(x)\,dx \\
    & \ap \sum_{k = - \infty}^M 2^{k \b}\bigg( \sup_{y \in (x_k,\infty)} U(y)^{-1} \sup_{t \in (0,y)} U(t)h(t)\bigg)^{\b} \\
    & \ap \sum_{k = - \infty}^M 2^{k \b}\bigg( \sup_{y \in (x_k,x_{k+1})} U(y)^{-1} \sup_{t \in (0,y)} U(t)h(t)\bigg)^{\b}.
    \end{align*}
    For every $-\infty < k \le M$ there exists $y_k \in (x_k,x_{k+1})$ such that
    $$
    \sup_{y \in (x_k,x_{k+1})} U(y)^{-1} \sup_{t \in (0,y)} U(t)h(t) \le 2 U(y_k)^{-1} \sup_{t \in (0,y_k)} U(t)h(t).
    $$
    Therefore,
    \begin{align*}
    B_1 & \ls \sum_{k = - \infty}^M 2^{k \b}\bigg( U(y_k)^{-1} \sup_{t \in (0,y_k)} U(t)h(t)\bigg)^{\b} \\
    & \ap \sum_{k = - \infty}^M 2^{k \b}\bigg( U(y_k)^{-1} \sup_{t \in (0,y_{k-2})} U(t)h(t)\bigg)^{\b} + \sum_{k = - \infty}^M 2^{k \b}\bigg( U(y_k)^{-1} \sup_{t \in (y_{k-2},y_k)} U(t)h(t)\bigg)^{\b} = : I + II.
    \end{align*}
    Note that $2^k \le \int_0^{y_k} g(x)\,dx \le 2^{k+1}$ and $2^{k-1} \le \int_{y_{k-2}}^{y_k} g(x)\,dx \le 2^{k+1}$, $- \infty < k \le M$. It yields that
    \begin{align*}
    I & \ls  \sum_{k = - \infty}^M \int_{y_{k-2}}^{y_k} \bigg( \int_x^{y_k} g(t)\,dt\bigg)^{\b - 1} g(x)\,dx
    \cdot \bigg( U(y_k)^{-1} \sup_{t \in (0,y_{k-2})} U(t)h(t)\bigg)^{\b} \\
    & \le \sum_{k = - \infty}^M \int_{y_{k-2}}^{y_k} \bigg( \int_x^{y_k} U(t)^{-1}g(t)\,dt\bigg)^{\b - 1} U(x)^{-1}g(x)\,dx
    \cdot \bigg( \sup_{t \in (0,y_{k-2})} U(t)h(t)\bigg)^{\b} \\
    & \le \sum_{k = - \infty}^M \int_{y_{k-2}}^{y_k} \bigg( \int_x^{\infty} U(t)^{-1}g(t)\,dt\bigg)^{\b - 1} \bigg( \sup_{t \in (0,x)} U(t)h(t)\bigg)^{\b} U(x)^{-1}g(x)\,dx
    \\
    & \ls \int_0^{\infty} \bigg( \int_x^{\infty} U(\tau)^{-1} g(\tau)\,d\tau\bigg)^{\b - 1}  \bigg( \sup_{t \in (0,x)} U(t)h(t)\bigg)^{\b}\,U(x)^{-1}g(x)\,dx = A_2.
    \end{align*}
    For $II$ we have that
    \begin{align*}
    II & \ls  \sum_{k = - \infty}^M \int_{y_{k-4}}^{y_{k-2}} \bigg( \int_{y_{k-4}}^x g(t)\,dt\bigg)^{\b - 1} g(x)\,dx
    \cdot \bigg( U(y_k)^{-1} \sup_{t \in (y_{k-2},y_k)} U(t)h(t)\bigg)^{\b} \\
    & \le \sum_{k = - \infty}^M  \int_{y_{k-4}}^{y_{k-2}} \bigg( \int_0^x g(t)\,dt\bigg)^{\b - 1} g(x)\,dx
    \cdot \, \bigg(\sup_{t \in (y_{k-2},\infty)} h(t)\bigg)^{\b} \\
    & \le \sum_{k = - \infty}^M  \int_{y_{k-4}}^{y_{k-2}} \bigg( \int_0^x g(t)\,dt\bigg)^{\b - 1}  \bigg(\sup_{t \in (x,\infty)} h(t)\bigg)^{\b} g(x)\,dx
    \\
    & \ls \int_0^{\infty} \bigg( \int_0^x g(t)\,dt\bigg)^{\b - 1} \bigg(\sup_{t \in (x, \infty)} h(t)\bigg)^{\b} \,g(x)dx = A_1.
    \end{align*}
    Combining, we get that $B_1 \ls A_1 + A_2$.

    Now we show that $B_2 \ls A_1 + A_2$. We will consider the case when $\int_0^{\infty} U(t)^{-1}g(t)\,dt < \infty$ (It is much simpler to deal with the case when $\int_0^{\infty} U(t)^{-1}g(t)\,dt = \infty$). Define  a sequence $\{x_k\}_{k = N}^{\infty}$ such that $2^{-k} = \int_{x_k}^{\infty} U(\tau)^{-1} g(\tau)\,d\tau$ if $N \le k < \infty$ and $2^{-N} < \int_0^{\infty} U(\tau)^{-1} g(\tau)\,d\tau \le 2^{- N + 1}$.
    By using elementary calculations, we find that
    \begin{align*}
    B_2 & \le \int_0^{\infty} \bigg( \int_x^{\infty} U(\tau)^{-1} g(\tau)\,d\tau\bigg)^{\b - 1}  \bigg(\sup_{y \in (0,x)} U(y)\sup_{t \in (y,\infty)} h(t)\bigg)^{\b}\,U(x)^{-1}g(x)dx \\
    & \ap \sum_{k = N}^{\infty} 2^{- k \b} \bigg(\sup_{y \in (0,x_k)} U(y)\sup_{t \in (y,\infty)} h(t)\bigg)^{\b} \\
    & \ap \sum_{k = N}^{\infty} 2^{- k \b} \bigg(\sup_{y \in (x_{k-1},x_k)} U(y)\sup_{t \in (y,\infty)} h(t)\bigg)^{\b}.
    \end{align*}
    For every $k = N,\, N+1,\ldots$ there exists $y_k \in (x_{k-1},x_k)$ such that
    $$
    \sup_{y \in (x_{k-1},x_k)} U(y)\sup_{t \in (y,\infty)} h(t) \le 2 U(y_k) \sup_{t \in (y_k,\infty)} h(t).
    $$
    Hence
    \begin{align*}
    B_2 & \ls \sum_{k = N}^{\infty} 2^{- k \b} \bigg( U(y_k) \sup_{t \in (y_k,\infty)} h(t) \bigg)^{\b} \\
    & \ap \sum_{k = N}^{\infty} 2^{- k \b} \bigg( U(y_k) \sup_{t \in (y_k,y_{k+2})} h(t) \bigg)^{\b} +  \sum_{k = N}^{\infty} 2^{- k \b} \bigg( U(y_k) \sup_{t \in (y_{k+2},\infty)} h(t)\bigg)^{\b} = : III + IV.
    \end{align*}
    Since $2^{-k-1} \le \int_{y_k}^{\infty} U(\tau)^{-1} g(\tau)\,d\tau \le 2^{-k}$ and $2^{-k-2} \le \int_{y_k}^{y_{k+2}} U(\tau)^{-1} g(\tau)\,d\tau \le 2^{-k}$, $k=N,\,N+1,\ldots$, we have that
    \begin{align*}
    III & \ls \sum_{k = N}^{\infty} \int_{y_{k+2}}^{y_{k+4}} \bigg( \int_x^{y_{k+4}} U(\tau)^{-1} g(\tau)\,d\tau\bigg)^{\b - 1}  \,U(x)^{-1}g(x)\,dx \cdot \bigg(  U(y_k) \sup_{t \in (y_k,y_{k+2})} h(t) \bigg)^{\b} \\
    & \le \sum_{k = N}^{\infty} \int_{y_{k+2}}^{y_{k+4}} \bigg( \int_x^{\infty} U(\tau)^{-1} g(\tau)\,d\tau\bigg)^{\b - 1}  \bigg( \sup_{t \in (0,x)} U(t)h(t) \bigg)^{\b}\,U(x)^{-1}g(x)\,dx   \\
    & \ls \int_0^{\infty} \bigg( \int_x^{\infty} U(\tau)^{-1} g(\tau)\,d\tau\bigg)^{\b - 1}  \bigg( \sup_{t \in (0,x)} U(t)h(t)\bigg)^{\b}\,U(x)^{-1}g(x)\,dx \ap A_2.
    \end{align*}
    Moreover,
    \begin{align*}
    IV & \ls \sum_{k = N}^{\infty} \int_{y_k}^{y_{k+2}} \bigg( \int_{y_k}^x U(\tau)^{-1} g(\tau)\,d\tau\bigg)^{\b - 1}  \,U(x)^{-1}g(x)dx \cdot \bigg( U(y_k) \sup_{t \in (y_{k+2},\infty)} h(t)\bigg)^{\b} \\
    & \le \sum_{k = N}^{\infty} \int_{y_k}^{y_{k+2}} \bigg( \int_{y_k}^x  g(\tau)\,d\tau\bigg)^{\b - 1}  g(x)\,dx \cdot \bigg( \sup_{t \in (y_{k+2},\infty)} h(t)\bigg)^{\b} \\
    & \le \sum_{k = N}^{\infty} \int_{y_k}^{y_{k+2}} \bigg( \int_0^x  g(\tau)\,d\tau\bigg)^{\b - 1} \bigg( \sup_{t \in (x,\infty)} h(t)\bigg)^{\b} g(x)\,dx \\
    & \ls \int_0^{\infty} \bigg( \int_0^x g(t)\,dt\bigg)^{\b - 1} \bigg(\sup_{t \in (x, \infty)} h(t)\bigg)^{\b} \,g(x)dx \ap A_1.
    \end{align*}
    Therefore, we obtain $B_2 \ls A_1 + A_2$. The proof is complete.
\end{proof}

\begin{thm}\label{maintheorem4}
Let $0<q_1, q_2<\i$, and $0 < p<q_2$. Let $v_1, v_2\in \W\I$, $u_1 \in \dual{\O_{q_1}}$ and $u_2\in \O_{q_2}$.
Assume that $v\in W\I \cap C\I$ and $0< \|u_2^{-1}\|_{q_2 \rw p,(x,\infty)}<\i,\, x > 0$.

{\rm (i)} If $q_1\le q_2$, then
\begin{align*}
\|\Id\|_{\cop_{p_1,q_1}(u_1,v_1)\rw \ces_{p_2,q_2}(u_2,v_2)} \ap & \, \sup_{x\in \I} \vp_2(x) \, \sup_{t\in (0,\infty)} {\mathcal V} (t,x) \,\|u_2\|_{q_2,(t,\infty)} \\
& + \|u_1\|_{q_1,\I}^{-1} \sup_{t \in \I} V(t) \|u_2\|_{q_2,(t,\infty)}.
\end{align*}

{\rm (ii)} If $q_2 < q_1$, then
\begin{align*}
\|\Id\|_{\cop_{p_1,q_1}(u_1,v_1) \rw \ces_{p_2,q_2}(u_2,v_2)} & \\
& \hspace{-3.5cm} \ap \bigg( \int_0^\i \vp_2(x)^{\frac{q_1 \rw q_2 \cdot q_1 \rw p}{q_2 \rw p}} V(x)^{q_1 \rw p}  \bigg(\sup_{t\in(0,\infty)} {\mathcal V}(t,x) \|u_2\|_{q_2,(t,\infty)}\bigg)^{q_1 \rw q_2} d \bigg( - \|u_1\|_{q_1,(0,x)}^{- q_1 \rw p}\bigg) \bigg)^{\frac{1}{q_1 \rw q_2}}\\
& \hspace{-3cm} + \|u_1\|_{q_1,\I}^{-1} \sup_{t \in \I} V(t) \|u_2\|_{q_2,(t,\infty)}.
\end{align*}
\end{thm}
\begin{proof}
By  Lemma \ref{mainlemma}, applying Theorem~[\ref{Emb-Cop-Lp}, (ii)], we get that
\begin{align*}
\|\Id\|_{\cop_{p_1,q_1}(u_1,v_1) \rw \ces_{p_2,q_2}(u_2,v_2)} \ap &
\|u_1\|_{q_1,(0,\I}^{-1}
\left\{\sup_{g\in \M^+\I} \ddfrac{ \|H^*g\|_{\infty,v^p,\I}} {\|g\|_{\frac{q_2}{q_2-p},u_2^{-p},\I}} \right\}^{\frac{1}{p}}\\
& + \left\{\sup_{g\in \M^+\I} \ddfrac{\bigg(\int_0^\i \|H^*g\|_{\i,v^p,(0,t)}^{\frac{q_1}{q_1-p}}
d \bigg(- \|u_1\|_{q_1,(0,t)}^{-\frac{q_1 p}{q_1 - p}}\bigg)\bigg)^{\frac{q_1-p}{q_1}}}
{\|g\|_{\frac{q_2}{q_2-p},u_2^{-p},\I}} \right\}^{\frac{1}{p}}\\
:= & C_3 + C_4.
\end{align*}
Note that
$$
C_3 = \|u_1\|_{q_1,\I}^{-1} \left[\|\Id\|_{L_{\frac{q_2}{q_2-p}}\big(u_2^{-p}\big)\rw \cop_{1, \i}(v^p, {\bf 1})}\right]^{\frac{1}{p}}.
$$
Using Theorem~[\ref{Emb-Lp-Cop}, (i)], we have that
\begin{equation*}
C_3 \ap \|u_1\|_{q_1,\I}^{-1} \sup_{t \in \I} V(t) \|u_2\|_{q_2,(t,\infty)}.
\end{equation*}

\item[(i)] Let  $q_1 \le q_2$, then Theorem~[\ref{IterHarSupremal.1}, (i)] yields that
\begin{equation*}
C_4 \ap \sup_{x\in \I} \bigg(\int_0^\i [{\mathcal V}(x,t)V(t)]^{q_1 \rw p} d \bigg( - \|u_1\|_{q_1,(0,t)}^{- q_1 \rw p}\bigg) \bigg)^{\frac{1}{q_1 \rw p}} \|u_2\|_{q_2,(x,\infty)}.
\end{equation*}
Since $\vp_2 / V$ is equivalent to decreasing function we have that
\begin{align*}
\sup_{x\in \I} \vp_2(x) \|u_2\|_{q_2,(x,\infty)} & = \sup_{x\in \I} \vp_2(x) V(x)^{-1}\, \sup_{t\in (0,x)} V(t) \,\|u_2\|_{q_2,(t,\infty)} \\
& = \sup_{x\in \I} \vp_2(x) \, \sup_{t\in (0,\infty)} {\mathcal V} (t,x) \,\|u_2\|_{q_2,(t,\infty)}, \qquad x > 0.
\end{align*}

\item[(ii)] Let  $q_2 < q_1$, then Theorem~[\ref{IterHarSupremal.1}, (ii)] yields that
\begin{align*}
C_4 \ap & \bigg(\int_0^{\infty}\bigg(\int_x^\i d \bigg(- \|u_1\|_{q_1,(0,t)}^{- q_1 \rw p}\bigg)\bigg)^{\frac{q_1 \rw q_2}{q_2 \rw p}} \bigg( \sup_{0 < \tau \le x} V(\tau)  \|u_2\|_{q_2,(\tau,\infty)}\bigg)^{q_1 \rw q_2} d \bigg( - \|u_1\|_{q_1,(0,x)}^{- q_1 \rw p}\bigg) \bigg)^{\frac{1}{q_1 \rw q_2}} \\
& +\bigg(\int_0^\i \bigg(\int_0^x V(t)^{q_1 \rw p} d \bigg( - \|u_1\|_{q_1,(0,t)}^{- q_1 \rw p}\bigg)\bigg)^{\frac{q_1 \rw q_2}{q_2 \rw p}} V(x)^{q_1 \rw p} \|u_2\|_{q_2,(x,\infty)}^{q_1 \rw q_2} d \bigg( - \|u_1\|_{q_1,(0,x)}^{- q_1 \rw p}\bigg)\bigg)^{\frac{1}{q_1 \rw q_2}}\\
\ap & \bigg( \int_0^\i \vp_2(x)^{\frac{q_1 \rw q_2 \cdot q_1 \rw p}{q_2 \rw p}} V(x)^{q_1 \rw p}  \bigg(\sup_{t\in(0,\infty)} {\mathcal V}(t,x) \|u_2\|_{q_2,(t,\infty)}\bigg)^{q_1 \rw q_2} d \bigg( - \|u_1\|_{q_1,(0,x)}^{- q_1 \rw p}\bigg) \bigg)^{\frac{1}{q_1 \rw q_2}}.
\end{align*}
In the last equivalence we have used Lemma \ref{gluing.lem} with
$$
u(x) = V(x)^{q_1 \rw p - 1}v(x), ~ g(t)dt = V(t)^{q_1 \rw p}d \bigg(- \|u_1\|_{q_1,(0,t)}^{- q_1 \rw p}\bigg),~ \b = \frac{q_1 \rw q_2}{q_1 \rw p} ~\mbox{and}~ h(t) = \|u_2\|_{q_2,(t,\infty)}^{q_1 \rw p}.
$$
It is clear that $U(x) \ap V(x)^{q_1 \rw p}$ and ${\mathcal U}(x,t) \ap {\mathcal V}(x,t)^{q_1 \rw p}$.
\end{proof}

\end{document}